\documentclass[12pt, reqno]{amsart}
\pdfoutput=1

\usepackage[margin=2cm]{geometry}
\usepackage{amsmath, amssymb, amsfonts, amstext, verbatim, amsthm, mathrsfs}
\usepackage{microtype}
\usepackage[all,arc,knot,poly,cmtip]{xy}
\usepackage{mathtools}
\usepackage[colorlinks=true,linkcolor=blue,citecolor=blue,urlcolor=blue,citebordercolor={0 0 1},urlbordercolor={0 0 1},linkbordercolor={0 0 1}]{hyperref} 
\usepackage{enumerate}
\usepackage{stmaryrd}
\usepackage{subcaption}
\usepackage{tikz}
\usepackage{setspace}
\usepackage{graphicx}
\theoremstyle{plain}
\newtheorem{theorem}{Theorem}[section]
\newtheorem{proposition}[theorem]{Proposition}
\newtheorem{lemma}[theorem]{Lemma}

\newtheorem{conjecture}[theorem]{Conjecture}

\theoremstyle{definition}
\newtheorem{definition}[theorem]{Definition}
\newtheorem{remark}[theorem]{Remark}

\DeclareMathOperator{\Spec}{\operatorname{Spec}}
\DeclareMathOperator{\Gr}{\operatorname{Gr}}
\DeclareMathOperator{\diag}{\operatorname{diag}}
\DeclareMathOperator{\Hom}{\operatorname{Hom}}
\DeclareMathOperator{\Tr}{\operatorname{Tr}}
\DeclareMathOperator{\Sk}{\operatorname{Sk}}
\DeclareMathOperator{\End}{\operatorname{End}}

\begin{document}

\title{Skein algebras and quantized Coulomb branches}
\author{Dylan G.L. Allegretti}
\author{Peng Shan}
\address{Yau Mathematical Sciences Center, Department of Mathematical Sciences, Tsinghua~University, 100084, Beijing, China}
\email{dylanallegretti@tsinghua.edu.cn}
\email{pengshan@tsinghua.edu.cn}

\date{}

\maketitle

\begin{abstract}
To a compact oriented surface of genus at most one with boundary, we associate a quantized $K$-theoretic Coulomb branch in the sense of Braverman, Finkelberg, and Nakajima. In~the case where the surface is a three- or four-holed sphere or a one-holed torus, we describe a relationship between this quantized Coulomb branch and the Kauffman bracket skein algebra of the surface. We formulate a general conjecture relating these algebras.
\end{abstract}

\setcounter{tocdepth}{1}
\tableofcontents

\section{Introduction}

Character varieties of surfaces are fundamental objects in modern mathematics, appearing in low-dimensional topology, representation~theory, and mathematical~physics, among other areas. Given a reductive algebraic group~$G$, the $G$-character variety of a surface is an affine scheme parametrizing principal $G$-bundles with flat connection on the surface. In the case where $G=\mathrm{SL}_2$, the character variety has a well known quantization given by a noncommutative algebra called the Kauffman~bracket skein~algebra~\cite{BFK99,PS00,T91}. In this paper, we propose a new approach to studying the $\mathrm{SL}_2$-character variety of a surface and its quantization defined by the skein~algebra.

Our perspective is based on ideas coming from mathematical physics, specifically the study of four-dimensional $N=2$ field theories of class~S. These are quantum field theories, first constructed by Gaiotto~\cite{G12} and famously studied by Gaiotto, Moore, and Neitzke~\cite{GMN13a,GMN13b}, whose properties are encoded in the geometry of surfaces. In particular, one of the important insights of these authors was that the Coulomb~branch of a class~S theory compactified on a circle should be closely related to a character variety of the surface that defines the theory.

The notion of a Coulomb branch originates in physics, where it functions as a kind of parameter space for a supersymmetric quantum field theory. In a groundbreaking series of papers, Braverman, Finkelberg, and Nakajima have introduced a rigorous mathematical definition of this object for a certain class of theories~\cite{BFN18,BFN19a,N16}. Their approach is to first define its coordinate ring as a homology group with a convolution product and then take the spectrum and study its geometric properties. The present paper is concerned with a variant of this construction in which the coordinate ring is constructed by taking $K$-theory rather than homology. The resulting $K$-theoretic Coulomb branch has a natural quantization, described explicitly in some cases in~\cite{FT19a,FT19b}.

Starting from a compact oriented surface of genus at most one with boundary, we show that the methods of Braverman, Finkelberg, and Nakajima can be used to define an associated quantized $K$-theoretic Coulomb~branch. In the special case where the surface is a three- or four-holed sphere or a one-holed torus, we relate this quantized Coulomb~branch to the skein algebra of the surface. Based on our observations, we conjecture a general relationship between these algebras. We expect that the relationship we describe here between skein algebras and quantized Coulomb~branches will provide powerful new geometric, categorical, and field theoretic tools for studying character varieties of surfaces.

\subsection{Two quantized algebras}

Let us describe in more detail the noncommutative algebras that we wish to relate. To any compact oriented surface~$S$, we can associate the $\mathrm{SL}_2$-character variety $\mathcal{M}_{\mathrm{flat}}(S,\mathrm{SL}_2)$. A complex point of this character variety is a principal $\mathrm{SL}_2(\mathbb{C})$-bundle with flat connection on~$S$. Given such a bundle with connection, we get a group homomorphism $\rho:\pi_1(S)\rightarrow\mathrm{SL}_2(\mathbb{C})$ sending the homotopy class of a loop~$\gamma$ on~$S$ to the monodromy of the flat connection around~$\gamma$. In this way, we identify the set of complex points with a space of group homomorphisms $\rho:\pi_1(S)\rightarrow\mathrm{SL}_2(\mathbb{C})$ up to conjugation. For each homotopy class of loops~$\gamma$ on~$S$, there is a regular function~$\Tr_\gamma$ on the character variety mapping a homomorphism $\rho$ to~$\Tr\rho(\gamma)$. These functions generate the coordinate ring of~$\mathcal{M}_{\mathrm{flat}}(S,\mathrm{SL}_2)$ as an algebra.

When $S$ has nonempty boundary, it is natural to consider the subscheme of $\mathcal{M}_{\mathrm{flat}}(S,\mathrm{SL}_2)$ that we get by prescribing the conjugacy class of the monodromy around each boundary component of~$S$. The conjugacy class of a semisimple $2\times2$ matrix of determinant one with complex coefficients is determined by its trace. If $\gamma_1,\dots,\gamma_n$ are curves freely homotopic to the boundary components of~$S$, we therefore specify a vector $\boldsymbol{\lambda}=(\lambda_1,\dots,\lambda_n)\in(\mathbb{C}^*)^n$ and consider the closed subscheme $\mathcal{M}_{\mathrm{flat}}^{\boldsymbol{\lambda}}(S,\mathrm{SL}_2)\subset\mathcal{M}_{\mathrm{flat}}(S,\mathrm{SL}_2)$ cut out by the relations 
\[
\Tr_{\gamma_i}=\lambda_i+\lambda_i^{-1}, \quad i=1,\dots,n.
\]
The numbers~$\lambda_i^{\pm1}$ thus specify the eigenvalues of the monodromy around~$\gamma_i$.

The Kauffman bracket skein algebra~$\Sk_A(S)$ is a noncommutative algebra that quantizes the $\mathrm{SL}_2$-character variety. As we will explain in Section~\ref{sec:TheKauffmanBracketSkeinAlgebra}, it is a $\mathbb{C}[A^{\pm1}]$-algebra generated by isotopy classes of framed links in the three-manifold $S\times[0,1]$, subject to the same skein relations used to define the Kauffman bracket in knot theory. Specializing the variable $A$ to~$-1$, we get an isomorphism 
\[
\Sk_A(S)\otimes_{\mathbb{C}[A^{\pm1}]}\left(\mathbb{C}[A^{\pm1}]/(A+1)\right)\cong\mathcal{O}(\mathcal{M}_{\mathrm{flat}}(S,\mathrm{SL}_2))
\]
with the coordinate ring of the $\mathrm{SL}_2$-character variety. When the surface~$S$ has nonempty boundary, we will be interested in a variant of the skein algebra denoted $\Sk_{A,\boldsymbol{\lambda}}(S)$. The latter is an algebra over a Laurent polynomial ring $\mathbb{C}[A^{\pm1},\lambda_1^{\pm1},\dots,\lambda_n^{\pm1}]$. By specializing $A$ to~$-1$ and specializing the variables~$\lambda_i$ to nonzero complex numbers, we recover the coordinate ring of the subscheme $\mathcal{M}_{\mathrm{flat}}^{\boldsymbol{\lambda}}(S,\mathrm{SL}_2)\subset\mathcal{M}_{\mathrm{flat}}(S,\mathrm{SL}_2)$.

Our goal in this paper is to relate the skein algebra to an associated quantized Coulomb~branch. In~general, a Coulomb~branch is a space that physicists associate to a compact Lie~group~$G_{\mathrm{cpt}}$, called the gauge group, and a quaternionic representation~$M$ of~$G_{\mathrm{cpt}}$. There are by now several approaches to defining these spaces mathematically under various additional assumptions~\cite{BDFRT22,BFN18,T22}. In the present paper, we adopt the original approach of Braverman, Finkelberg, and Nakajima~\cite{BFN18}. Namely, we replace $G_{\mathrm{cpt}}$ by its complexification~$G$ and assume that the representation~$M$ is of cotangent type, that is, $M\cong N\oplus N^*$ for some complex representation~$N$. Under this assumption, Braverman, Finkelberg, and Nakajima give a precise mathematical definition of the Coulomb branch as an affine scheme.

Their approach involves a space $\mathcal{R}_{G,N}$ called the variety of triples. It is an ind-scheme closely related to the affine Grassmannian of~$G$. We review its definition in Section~\ref{sec:QuantizedKTheoreticCoulombBranches}. In the examples that we consider, there is an algebraic torus~$F$, called a flavor symmetry group, acting on~$N$ so that this vector space~$N$ becomes a representation of~$\widetilde{G}=G\times F$. There is a natural action of the group~$\widetilde{G}_\mathcal{O}$ of $\mathcal{O}$-valued points of~$\widetilde{G}$ on the variety of triples where $\mathcal{O}=\mathbb{C}\llbracket z\rrbracket$. Following the approach of~\cite{BFN18,VV10}, one can define the $\widetilde{G}_{\mathcal{O}}$-equivariant $K$-theory $K^{\widetilde{G}_{\mathcal{O}}}(\mathcal{R}_{G,N})$ of the variety of triples. It has a convolution product~$*$ making it into a commutative algebra, and the $K$-theoretic Coulomb~branch of~$(\widetilde{G},N)$ is defined as the spectrum 
\[
\mathcal{M}_{\mathrm{C}}(\widetilde{G},N)=\Spec\left(K^{\widetilde{G}_{\mathcal{O}}}(\mathcal{R}_{G,N}),*\right).
\]
As we will review below, there is a natural $\mathbb{C}^*$-action on the variety of triples $\mathcal{R}_{G,N}$ by loop rotation, and so we can consider the equivariant $K$-theory $K^{\widetilde{G}_{\mathcal{O}}\rtimes\mathbb{C}^*}(\mathcal{R}_{G,N})$. There is again a convolution product on this space, which makes it into a noncommutative algebra over $K^{\mathbb{C}^*}(\mathrm{pt})\cong\mathbb{C}[q^{\pm1}]$ called the quantized $K$-theoretic Coulomb branch of~$(\widetilde{G},N)$, and one recovers the algebra $K^{\widetilde{G}_{\mathcal{O}}}(\mathcal{R}_{G,N})$ by specializing $q$ to~1. We replace~$\mathbb{C}^*$ by its double cover and thus extend scalars to~$\mathbb{C}[q^{\pm\frac{1}{2}}]$.

\subsection{Summary of the main results}

Let us consider the surface~$S=S_{g,n}$ obtained from a closed oriented surface of genus~$g$ by removing $n$ open disks with embedded disjoint closures. We assume that $S$ has negative Euler~characteristic $\chi(S)=2-2g-n<0$. In Section~\ref{sec:ProofsOfTheMainResults}, we explain how to associate, to any pants decomposition of~$S$, corresponding gauge and flavor symmetry groups 
\[
G\cong\mathrm{SL}_2(\mathbb{C})^{3g-3+n}, \quad F\cong(\mathbb{C}^*)^n.
\]
There is a representation $M$ of $\widetilde{G}=G\times F$, defined as a direct sum of subspaces $M_P\cong\mathbb{C}^8$ for each pair of pants $P$ in the pants decomposition. This construction was essentially described in~\cite{G12}, though we replace the gauge group $\mathrm{SU}(2)^{3g-3+n}$ used there by its complexification. When the surface has genus $g\leq1$, we show that there exists a pants decomposition for which the corresponding representation $M$ is of cotangent type, that is, $M\cong N\oplus N^*$ for a complex representation $N$ of~$G$. Such a pants decomposition therefore determines an associated quantized Coulomb branch by the construction of Braverman, Finkelberg, and Nakajima~\cite{BFN18}.

The simplest example is when $S=S_{0,3}$ is a three-holed sphere. In this case the skein algebra of~$S$ is commutative. There is a unique pants decomposition with a single pair of pants, and we prove the following result.

\begin{theorem}
\label{thm:introS03main}
Let $(\widetilde{G},N)$ be the group and representation associated to the three-holed sphere~$S_{0,3}$. Then there is a $\mathbb{C}$-algebra isomorphism 
\[
\Sk_{A,\boldsymbol{\lambda}}(S_{0,3})\cong K^{\widetilde{G}_{\mathcal{O}}\rtimes\mathbb{C}^*}(\mathcal{R}_{G,N}).
\]
\end{theorem}

A more nontrivial example is when $S=S_{0,4}$ is a four-holed sphere. In this case the skein algebra of~$S$ is known to be closely related to the spherical DAHA of type~$(C_1^\vee,C_1)$; see~\cite{BS16,BS18,H19,O04,T13}. The associated class~S quantum field theory is a fundamental example of an $N=2$ field theory, which was famously studied by Seiberg and~Witten in~\cite{SW94}.

\begin{theorem}
\label{thm:introS04main}
Let $(\widetilde{G},N)$ be the group and representation associated to a pants decomposition of the four-holed sphere~$S_{0,4}$. Then there is a $\mathbb{C}$-algebra isomorphism 
\[
\Sk_{A,\boldsymbol{\lambda}}(S_{0,4})\cong K^{\widetilde{G}_{\mathcal{O}}\rtimes\mathbb{C}^*}(\mathcal{R}_{G,N}).
\]
\end{theorem}

A third important example is when $S=S_{1,1}$ is a one-holed torus. In this case the skein algebra is closely related to the spherical DAHA of type~$A_1$; see~\cite{H19,S19}. The associated class~S field theory is another a fundamental example of an $N=2$ quantum field theory. It is the so called 4D~$N=2^*$ theory studied by Donagi and Witten in~\cite{DW94}. We refer to~\cite{GKNPS23} for a very detailed modern treatment of this theory and its relation to the DAHA.

\begin{theorem}
\label{thm:introS11main}
Let $(\widetilde{G},N)$ be the group and representation associated to a pants decomposition of the one-holed torus~$S_{1,1}$. Then there is a $\mathbb{Z}_2$-action on $\Sk_{A,\lambda}(S_{1,1})$ and a $\mathbb{C}$-algebra isomorphism 
\[
\Sk_{A,\lambda}(S_{1,1})^{\mathbb{Z}_2}\cong K^{\widetilde{G}_{\mathcal{O}}\rtimes\mathbb{C}^*}(\mathcal{R}_{G,N}).
\]
\end{theorem}

The appearance of the $\mathbb{Z}_2$-invariant subalgebra in Theorem~\ref{thm:introS11main} is consistent with predictions coming from the physics literature; see Section~8.4 of~\cite{GMN13b} and Section~4 of~\cite{GKNPS23}. We note that each of our theorems gives an isomorphism of integral forms that identifies the quantization parameter~$q$ in the quantized Coulomb branch with the element $A^{-2}$ in the skein algebra. We therefore get an isomorphism between classical limits of the algebras. We also note that the proofs of our Theorems~\ref{thm:introS04main} and~\ref{thm:introS11main} are based on the representation theory of spherical~DAHAs. While there are many known results relating DAHA to Coulomb~branches~\cite{BEF20,KN18}, equivariant $K$-theory~\cite{BFM05,VV10,V05}, and the skein algebras of various surfaces~\cite{AS19,CS21,GJV23,H19}, our results are the first attempt to relate skein algebras and quantized Coulomb branches in a precise and systematic way using~DAHA.

\subsection{Conjecture and generalizations}

Based on our observations in Theorems~\ref{thm:introS03main}, \ref{thm:introS04main}, and~\ref{thm:introS11main}, we propose the following conjecture on the relationship between skein~algebras and quantized Coulomb~branches.

\begin{conjecture}
\label{conj}
Let $S=S_{g,n}$ be a surface of genus $g\leq1$ with $n>0$ boundary components, and let $(\widetilde{G},N)$ be the group and representation associated to a pants decomposition of~$S$ as above. Then 
\begin{enumerate}
\item If $g=0$ there is a $\mathbb{C}$-algebra isomorphism 
\[
\Sk_{A,\boldsymbol{\lambda}}(S)\cong K^{\widetilde{G}_{\mathcal{O}}\rtimes\mathbb{C}^*}(\mathcal{R}_{G,N}).
\]
\item\label{conj:genus1} If $g=1$ there is a $\mathbb{Z}_2$-action on~$\Sk_{A,\boldsymbol{\lambda}}(S)$ and a $\mathbb{C}$-algebra isomorphism 
\[
\Sk_{A,\boldsymbol{\lambda}}(S)^{\mathbb{Z}_2}\cong K^{\widetilde{G}_{\mathcal{O}}\rtimes\mathbb{C}^*}(\mathcal{R}_{G,N}).
\]
\end{enumerate}
\end{conjecture}

In Section~\ref{sec:EvidenceFortheMainConjecture}, we discuss the expected properties of these isomorphisms in more detail and describe further evidence supporting this conjecture.

One can imagine many possible generalizations of our results beyond Conjecture~\ref{conj}. For example, we expect that our results can be generalized in some fashion to include the surfaces $S_{g,n}$ for any integers~$g,n\geq0$. It would be interesting to consider character varieties of such surfaces associated to higher~rank groups or allow irregular singularities. We believe that such generalizations will typically involve representations of noncotangent type. We also expect that there are interesting extensions of our results involving the full DAHA, and not just the spherical subalgebra considered in the present paper.

In each of our main theorems, the isomorphism we construct is associated to a specific pants decomposition of the surface. Given two different pants decompositions, one has a canonical isomorphism of the associated quantized Coulomb branches. It would be good to better understand this isomorphism. It would also be good to understand the action of the mapping class group on the quantized Coulomb branch. In physical language, this is related to the S-duality group of the associated quantum field theory. We explore related ideas in the companion paper~\cite{AS24}.

Another interesting problem is to study the behavior of the Coulomb branches when a surface is cut into pieces or when surfaces are glued along their boundaries. Braverman, Finkelberg, and Nakajima have obtained relevant results in this direction in~\cite{BFN19b}. They describe an alternative construction of Coulomb branches using ring objects in the equivariant derived category of the affine Grassmannian. Their Proposition~5.23 is a gluing formula that allows one to express the ring object associated to a surface in terms of the ring objects associated to pairs of pants in a pants decomposition. It would be interesting to develop a $K$-theoretic version of this result and use it to study character varieties of surfaces.

Finally, there is an interesting application of these ideas to categorification of skein algebras. In~\cite{T14}, D.~Thurston showed that the classical limit of the skein algebra possesses a canonical basis with positive structure constants. In Question~1.5 of the same paper, he asked whether this positivity property could be explained by the existence of a categorification of the skein algebra. One possible approach to this question was described in~\cite{Q22}. Alternatively, one can consider the derived category of $\widetilde{G}_{\mathcal{O}}\rtimes\mathbb{C}^*$-equivariant coherent sheaves on the variety of triples~$\mathcal{R}_{G,N}$. It is a triangulated monoidal category which categorifies the quantized Coulomb branch in the sense that its Grothendieck ring is the algebra $K^{\widetilde{G}_{\mathcal{O}}\rtimes\mathbb{C}^*}(\mathcal{R}_{G,N})$. Our results imply that this category provides a categorification of the skein algebra in some cases. The work of Cautis and Williams~\cite{CW23} provides powerful categorical methods for studying canonical bases of quantized Coulomb branches and establishing their positivity properties. We will discuss this in detail in~\cite{AKS25}.

\subsection*{Acknowledgements.}
We thank Haimiao~Chen, Renaud~Detcherry, Stephen~Doty, Sergei~Gukov, Hiraku~Nakajima, Du~Pei,  Dan~Xie, and Wenbin~Yan for helpful discussions and answers to questions. PS is supported by NSFC Grant No.~12225108.

\section{The Kauffman bracket skein algebra}
\label{sec:TheKauffmanBracketSkeinAlgebra}

In this section, we recall the definition of the Kauffman~bracket skein~algebra of a surface and its relationship with DAHA when the surface is a four-holed sphere or a one-holed torus. Most of the material in this section is either known or closely related to known results in the literature.

\subsection{The $\mathrm{SL}_2$-character variety}

Recall that if $S$ is a connected smooth manifold and $G$ is a Lie group, then a $G$-local system on~$S$ is a principal $G$-bundle on~$S$ with flat connection. To any such local system, we can associate a monodromy representation $\pi_1(S)\rightarrow G$, and in this way, we obtain a natural bijection between the set of isomorphism classes of $G$-local systems on~$S$ and conjugacy classes of group homomorphisms $\pi_1(S)\rightarrow G$.

In this paper, we will be interested in the case where $S$ is a compact oriented smooth surface with boundary and $G=\mathrm{SL}_2(\mathbb{C})$. In this case, the fundamental group $\pi_1(S)$ is finitely presented, and the set $\Hom(\pi_1(S),\mathrm{SL}_2)$ has the structure of an algebraic variety called the \emph{$\mathrm{SL}_2$-representation variety} of~$S$. The group $\mathrm{SL}_2$ acts on this variety by conjugation.

\begin{definition}
The \emph{$\mathrm{SL}_2$-character variety} of $S$ is the affine GIT~quotient 
\[
\mathcal{M}_{\mathrm{flat}}(S,\mathrm{SL}_2)\coloneqq\Hom(\pi_1(S),\mathrm{SL}_2)\sslash\mathrm{SL}_2
\]
of the representation variety by the conjugation action of~$\mathrm{SL}_2$.
\end{definition}

If $\Bbbk$ is any field, then a $\Bbbk$-point of the representation variety is the same as a group homomorphism $\rho:\pi_1(S)\rightarrow\mathrm{SL}_2(\Bbbk)$. If $\gamma\in\pi_1(S)$ then we get a matrix $\rho(\gamma)\in\mathrm{SL}_2(\Bbbk)$, and the trace $\Tr_\gamma(\rho)\coloneqq\Tr\rho(\gamma)\in\Bbbk$ is invariant under conjugation by~$\mathrm{SL}_2(\Bbbk)$. It is also invariant when we change the orientation of~$\gamma$ or the basepoint used to define the fundamental group. Thus we can associate, to any free homotopy class~$\gamma$ of unoriented closed curves on~$S$, a regular function $\Tr_\gamma$ on the character variety. It is known that these functions generate the coordinate ring of $\mathcal{M}_{\mathrm{flat}}(S,\mathrm{SL}_2)$.

We will say that a simple closed curve $\gamma\subset S$ is \emph{peripheral} if it is freely homotopic in~$S$ to a component of~$\partial S$. Suppose $\partial_1,\dots,\partial_n$ are the boundary components of~$S$, and write~$\gamma_i$ for the homotopy class of peripheral curves corresponding to the boundary component~$\partial_i$. To formulate our results, we will need to consider the closed subscheme of the character variety obtained by prescribing the values of the functions~$\Tr_{\gamma_i}$. More precisely, for any vector $\boldsymbol{\lambda}=(\lambda_1,\dots,\lambda_n)\in\mathbb{C}^n$, we define the \emph{relative character variety} to be the subscheme $\mathcal{M}_{\mathrm{flat}}^{\boldsymbol{\lambda}}(S,\mathrm{SL}_2)\subset\mathcal{M}_{\mathrm{flat}}(S,\mathrm{SL}_2)$ cut out by the relations 
\[
\Tr_{\gamma_i}=\lambda_i+\lambda_i^{-1}, \quad i=1,\dots,n.
\]
In the special case where $n=1$ and $\boldsymbol{\lambda}=(\lambda)$, we can simply write $\mathcal{M}_{\mathrm{flat}}^{\lambda}(S,\mathrm{SL}_2)$ for the relative character variety.

\subsection{The skein algebra}

Consider any three-dimensional smooth manifold~$M$. By a \emph{link} in $M$, we will mean a compact unoriented one-dimensional submanifold $L\subset M$ without boundary. By a \emph{framing} for a link $L\subset M$, we will mean a continuous map $L\rightarrow TM$ assigning a nonzero tangent vector $v_p\in T_pM$ to each point $p\in L$ so that $v_p\not\in T_pL$. A link equipped with a framing will be called a \emph{framed link}. By an \emph{isotopy} of two framed links we will mean an ambient isotopy of the underlying links in~$M$ that preserves their framings.

Now let $S$ be an oriented smooth surface. We will be interested in the associated three-manifold defined as the product $S\times[0,1]$. A link $L\subset S\times[0,1]$ in this three-manifold is said to have the \emph{blackboard framing} if, for every point $p\in L$, the associated framing vector $v_p\in T_p(S\times[0,1])$ is tangent to the $[0,1]$ factor and points towards~1. Given any isotopy class of framed links in $S\times[0,1]$ we can find a representative framed link $L\subset S\times[0,1]$ that has the blackboard framing. We can then describe the framed link by drawing its projection to~$S$ and, if two points of~$L$ project to the same point of~$S$, indicating the ordering of their $[0,1]$-coordinates. Below we will always draw pictures of framed links in this way.

Let us write $\mathcal{L}_A(S)$ for the $\mathbb{C}[A^{\pm1}]$-module freely generated by the isotopy classes of framed links in~$S\times[0,1]$. This module can be equipped with a natural bilinear product operation defined as follows. Given framed links $L_1$,~$L_2\subset S\times[0,1]$, we rescale the $[0,1]$-coordinates so that these links lie in $S\times(0,\frac{1}{2})$ and $S\times(\frac{1}{2},1)$, respectively. Then the product $L_1L_2$ is defined as the union of the rescaled links in $S\times[0,1]$. This operation is well defined on isotopy classes and gives~$\mathcal{L}_A(S)$ the structure of a $\mathbb{C}[A^{\pm1}]$-algebra.

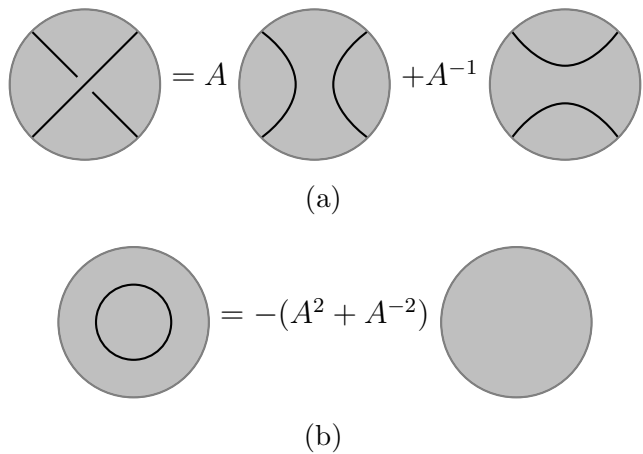
\begin{figure}[ht]
\begin{subfigure}{\textwidth}
\begin{center}
\begin{tikzpicture}
\draw[fill=gray,opacity=0.2] (0,0) circle (1);
\draw[black, thick] (-0.707,-0.707) -- (0.707,0.707);
\draw[black, thick] (-0.707,0.707) -- (-0.1,0.1);
\draw[black, thick] (0.1,-0.1) -- (0.707,-0.707);
\draw[gray, opacity=0.2, thick] (0,0) circle (1);
\end{tikzpicture}
\raisebox{1cm}{$= A$}
\begin{tikzpicture}
\draw[fill=gray,opacity=0.2] (0,0) circle (1);
\draw[black, thick] plot [smooth, tension=1] coordinates { (0.707,0.707) (0.25,0) (0.707,-0.707)};
\draw[black, thick] plot [smooth, tension=1] coordinates { (-0.707,0.707) (-0.25,0) (-0.707,-0.707)};
\draw[gray, opacity=0.2, thick] (0,0) circle (1);
\end{tikzpicture}
\raisebox{1cm}{$+ A^{-1}$}
\begin{tikzpicture}
\draw[fill=gray,opacity=0.2] (0,0) circle (1);
\draw[black, thick] plot [smooth, tension=1] coordinates { (-0.707,0.707) (0,0.25) (0.707,0.707)};
\draw[black, thick] plot [smooth, tension=1] coordinates { (-0.707,-0.707) (0,-0.25) (0.707,-0.707)};
\draw[gray, opacity=0.2, thick] (0,0) circle (1);
\end{tikzpicture}
\end{center}
\caption{\label{subfig:resolution}}
\bigskip
\end{subfigure}
\begin{subfigure}{\textwidth}
\begin{center}
\begin{tikzpicture}
\draw[fill=gray,opacity=0.2] (0,0) circle (1);
\draw[black, thick] (0,0) circle (0.5);
\draw[gray, opacity=0.2, thick] (0,0) circle (1);
\end{tikzpicture}
\raisebox{1cm}{$= -(A^2+A^{-2})$}
\begin{tikzpicture}
\draw[fill=gray,opacity=0.2] (0,0) circle (1);
\draw[gray, opacity=0.2, thick] (0,0) circle (1);
\end{tikzpicture}
\end{center}
\caption{\label{subfig:unknot}}
\bigskip
\end{subfigure}
\caption{The Kauffman bracket skein relations.\label{fig:skeinrelations}}
\end{figure}

We will consider a collection of relations between framed links depicted in Figure~\ref{fig:skeinrelations}. Each of the pictures appearing in these relations represents a framed link~$S\times[0,1]$. We depict only the portion of the link that projects to a neighborhood $U\subset S$, shaded in gray, and the links appearing in a given relation are assumed to be identical outside this neighborhood~$U$. The quotient of $\mathcal{L}_A(S)$ by these relations is called the \emph{(Kauffman bracket) skein algebra} of~$S$ and denoted $\Sk_A(S)$. It is known~\cite{BFK99,PS00,T91} that the skein algebra is a quantization of the $\mathrm{SL}_2$-character variety of~$S$. More~precisely, there is an isomorphism of algebras 
\[
\Sk_A(S)\otimes_{\mathbb{C}[A^{\pm1}]}\left(\mathbb{C}[A^{\pm1}]/(A+1)\right)\cong\mathcal{O}\left(\mathcal{M}_{\mathrm{flat}}(S,\mathrm{SL}_2)\right).
\]
If $L\subset S\times[0,1]$ is a link with the blackboard framing that projects to~$\gamma$, then this isomorphism sends $L$ to the regular function $-\Tr_\gamma$ on the character variety.

We will also need a relative version of the skein algebra of a compact surface $S$ with nonempty boundary. Suppose $\partial_1,\dots,\partial_n$ are the boundary components of~$S$, and define $\mathcal{L}_{A,\boldsymbol{\lambda}}(S)$ to be the $\mathbb{C}[A^{\pm1},\lambda_1^{\pm1},\dots,\lambda_n^{\pm1}]$-module freely generated by the set of isotopy classes of framed links in~$S\times[0,1]$. We define a product operation on this module exactly as we did for~$\mathcal{L}_A(S)$. For each boundary component~$\partial_i$, we consider the relation depicted in Figure~\ref{fig:puncturerelation}.

\begin{figure}[ht]
\begin{tikzpicture}
\draw[fill=gray,opacity=0.2,even odd rule] (0,0) circle (1) (0,0) circle (0.33);
\draw[black, thick] (0,0) circle (0.67);
\draw[gray, opacity=0.2, thick] (0,0) circle (1);
\draw[gray, opacity=0.2, thick] (0,0) circle (0.33);
\end{tikzpicture}
\raisebox{1cm}{$= -(\lambda_i+\lambda_i^{-1})$}
\begin{tikzpicture}
\draw[fill=gray,opacity=0.2,even odd rule] (0,0) circle (1) (0,0) circle (0.33);
\draw[gray, opacity=0.2, thick] (0,0) circle (1);
\draw[gray, opacity=0.2, thick] (0,0) circle (0.33);
\end{tikzpicture}
\caption{Additional relation associated to a boundary component.\label{fig:puncturerelation}}
\end{figure}
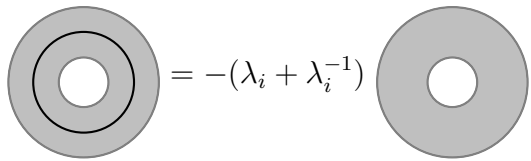

As in the previous relations, each picture in this relation represents a framed link in~$S\times[0,1]$. We depict only the portion of the link that projects to a neighborhood $U\subset S$ of the boundary component~$\partial_i$, shaded in gray, and the links on either side of the relation are assumed to be identical outside of~$U$. The quotient of $\mathcal{L}_{A,\boldsymbol{\lambda}}(S)$ by the relations depicted in Figures~\ref{fig:skeinrelations} and~\ref{fig:puncturerelation} will be denoted~$\Sk_{A,\boldsymbol{\lambda}}(S)$ and called the \emph{relative skein algebra} of~$S$. We will also need a slight variant of this construction. Namely, we consider the $\mathbb{C}$-algebra homomorphism $\mathbb{C}[A^{\pm1},\lambda_1^{\pm1},\dots,\lambda_n^{\pm1}]\rightarrow\mathbb{C}[A^{\pm1},\rho_1^{\pm1},\dots,\rho_n^{\pm1}]$ given by $A\mapsto A$ and $\lambda_i\mapsto\rho_i^2$ for all~$i$. This gives $\mathbb{C}[A^{\pm1},\rho_1^{\pm1},\dots,\rho_n^{\pm1}]$ the structure of a $\mathbb{C}[A^{\pm1},\lambda_1^{\pm1},\dots,\lambda_n^{\pm1}]$-module, and we write 
\[
\Sk_{A,\boldsymbol{\rho}}(S)=\mathbb{C}[A^{\pm1},\rho_1^{\pm1},\dots,\rho_n^{\pm1}]\otimes\Sk_{A,\boldsymbol{\lambda}}(S)
\]
for the algebra obtained from $\Sk_{A,\boldsymbol{\lambda}}(S)$ by extending scalars to~$\mathbb{C}[A^{\pm1},\rho_1^{\pm1},\dots,\rho_n^{\pm1}]$. If we have $n=1$, $\rho_1=\rho$, and $\lambda_1=\lambda$, we will denote the algebras $\Sk_{A,\boldsymbol{\lambda}}(S)$ and $\Sk_{A,\boldsymbol{\rho}}(S)$ by~$\Sk_{A,\lambda}(S)$ and~$\Sk_{A,\rho}(S)$, respectively.

\subsection{Presentations of skein algebras}

Let us write $S_{g,n}$ for the surface obtained from a closed oriented surface of genus~$g$ by removing $n$ open disks with embedded disjoint closures. In this subsection, we recall some presentations of the skein algebra of the four-holed sphere~$S_{0,4}$ and one-holed torus~$S_{1,1}$. Figure~\ref{fig:S04generators} shows a collection of closed curves on~$S=S_{0,4}$. We write~$\alpha$,~$\beta$,~$\gamma$, and~$\delta_i$ for links with the blackboard framing in $\Sk_A(S)$ that project to these curves as indicated in the figure.

\begin{figure}[ht]
\begin{subfigure}{\textwidth/5}
\begin{center}
\begin{tikzpicture}[scale=1.25]
\draw[gray, thin, rotate=45] (1,0) ellipse (0.05 and 0.25);
\draw[gray, thin, rotate=135] (1,0) ellipse (0.05 and 0.25);
\draw[gray, thin, rotate=-135] (1,0) ellipse (0.05 and 0.25);
\draw[gray, thin, rotate=-45] (1,0) ellipse (0.05 and 0.25);
\draw[gray, thin] plot [smooth, tension=1] coordinates { (0.53,0.89) (0,0.75) (-0.53,0.89)};
\draw[gray, thin] plot [smooth, tension=1] coordinates { (-0.89,0.53) (-0.75,0) (-0.89,-0.53)};
\draw[gray, thin] plot [smooth, tension=1] coordinates { (-0.53,-0.89) (0,-0.75) (0.53,-0.89)};
\draw[gray, thin] plot [smooth, tension=1] coordinates { (0.89,-0.53) (0.75,0) (0.89,0.53)};
\draw[black, thick, dotted] plot [smooth, tension=1] coordinates { (-0.75,0) (-0.5,0.1) (0.5,0.1) (0.75,0)};
\draw[black, thick] plot [smooth, tension=1] coordinates { (-0.75,0) (-0.5,-0.1) (0.5,-0.1) (0.75,0)};
\node at (0,-0.3) {$\alpha$};
\end{tikzpicture}
\end{center}
\caption{}
\end{subfigure}
\begin{subfigure}{\textwidth/5}
\begin{center}
\begin{tikzpicture}[scale=1.25]
\draw[gray, thin, rotate=45] (1,0) ellipse (0.05 and 0.25);
\draw[gray, thin, rotate=135] (1,0) ellipse (0.05 and 0.25);
\draw[gray, thin, rotate=-135] (1,0) ellipse (0.05 and 0.25);
\draw[gray, thin, rotate=-45] (1,0) ellipse (0.05 and 0.25);
\draw[gray, thin] plot [smooth, tension=1] coordinates { (0.53,0.89) (0,0.75) (-0.53,0.89)};
\draw[gray, thin] plot [smooth, tension=1] coordinates { (-0.89,0.53) (-0.75,0) (-0.89,-0.53)};
\draw[gray, thin] plot [smooth, tension=1] coordinates { (-0.53,-0.89) (0,-0.75) (0.53,-0.89)};
\draw[gray, thin] plot [smooth, tension=1] coordinates { (0.89,-0.53) (0.75,0) (0.89,0.53)};
\draw[black, thick, dotted] plot [smooth, tension=1] coordinates { (0,-0.75) (-0.1,-0.5) (-0.1,0.5) (0,0.75)};
\draw[black, thick] plot [smooth, tension=1] coordinates { (0,-0.75) (0.1,-0.5) (0.1,0.5) (0,0.75)};
\node at (0.35,0) {$\beta$};
\end{tikzpicture}
\end{center}
\caption{}
\end{subfigure}
\begin{subfigure}{\textwidth/5}
\begin{center}
\begin{tikzpicture}[scale=1.25]
\draw[gray, thin, rotate=45] (1,0) ellipse (0.05 and 0.25);
\draw[gray, thin, rotate=135] (1,0) ellipse (0.05 and 0.25);
\draw[gray, thin, rotate=-135] (1,0) ellipse (0.05 and 0.25);
\draw[gray, thin, rotate=-45] (1,0) ellipse (0.05 and 0.25);
\draw[gray, thin] plot [smooth, tension=1] coordinates { (0.53,0.89) (0,0.75) (-0.53,0.89)};
\draw[gray, thin] plot [smooth, tension=1] coordinates { (-0.89,0.53) (-0.75,0) (-0.89,-0.53)};
\draw[gray, thin] plot [smooth, tension=1] coordinates { (-0.53,-0.89) (0,-0.75) (0.53,-0.89)};
\draw[gray, thin] plot [smooth, tension=1] coordinates { (0.89,-0.53) (0.75,0) (0.89,0.53)};
\draw[black, thick] plot [smooth, tension=1] coordinates { (0,0.75) (-0.25,0.65) (-0.65,0.25) (-0.75,0)};
\draw[black, thick] plot [smooth, tension=1] coordinates { (0.75,0) (0.65,-0.25) (0.25,-0.65) (0,-0.75)};
\draw[black, thick, dotted] plot [smooth, tension=1] coordinates { (0.75,0) (0.65,0.25) (0.25,0.65) (0,0.75)};
\draw[black, thick, dotted] plot [smooth, tension=1] coordinates { (-0.75,0) (-0.65,-0.25) (-0.25,-0.65) (0,-0.75)};
\node at (0.3,-0.3) {$\gamma$};
\end{tikzpicture}
\end{center}
\caption{}
\end{subfigure}
\begin{subfigure}{\textwidth/5}
\begin{center}
\begin{tikzpicture}[scale=1.25]
\draw[gray, thin, rotate=45] (1,0) ellipse (0.05 and 0.25);
\draw[gray, thin, rotate=135] (1,0) ellipse (0.05 and 0.25);
\draw[gray, thin, rotate=-135] (1,0) ellipse (0.05 and 0.25);
\draw[gray, thin, rotate=-45] (1,0) ellipse (0.05 and 0.25);
\draw[gray, thin] plot [smooth, tension=1] coordinates { (0.53,0.89) (0,0.75) (-0.53,0.89)};
\draw[gray, thin] plot [smooth, tension=1] coordinates { (-0.89,0.53) (-0.75,0) (-0.89,-0.53)};
\draw[gray, thin] plot [smooth, tension=1] coordinates { (-0.53,-0.89) (0,-0.75) (0.53,-0.89)};
\draw[gray, thin] plot [smooth, tension=1] coordinates { (0.89,-0.53) (0.75,0) (0.89,0.53)};
\draw[black, thick, dotted] plot [smooth, tension=1] coordinates { (0.35,0.82) (0.62,0.62) (0.82,0.35)};
\draw[black, thick] plot [smooth, tension=1] coordinates { (0.35,0.82) (0.53,0.53) (0.82,0.35)};
\draw[black, thick, dotted] plot [smooth, tension=1] coordinates { (-0.82,-0.35) (-0.62,-0.62) (-0.35,-0.82)};
\draw[black, thick] plot [smooth, tension=1] coordinates { (-0.82,-0.35) (-0.53,-0.53) (-0.35,-0.82)};
\draw[black, thick, dotted] plot [smooth, tension=1] coordinates { (-0.35,0.82) (-0.62,0.62) (-0.82,0.35)};
\draw[black, thick] plot [smooth, tension=1] coordinates { (-0.35,0.82) (-0.53,0.53) (-0.82,0.35)};
\draw[black, thick, dotted] plot [smooth, tension=1] coordinates { (0.82,-0.35) (0.62,-0.62) (0.35,-0.82)};
\draw[black, thick] plot [smooth, tension=1] coordinates { (0.82,-0.35) (0.53,-0.53) (0.35,-0.82)};
\node at (-0.38,0.38) {$\delta_2$};
\node at (0.38,0.38) {$\delta_1$};
\node at (-0.38,-0.38) {$\delta_3$};
\node at (0.38,-0.38) {$\delta_4$};
\end{tikzpicture}
\end{center}
\caption{}
\end{subfigure}
\caption{Generators of $\Sk_A(S_{0,4})$.\label{fig:S04generators}}
\end{figure}
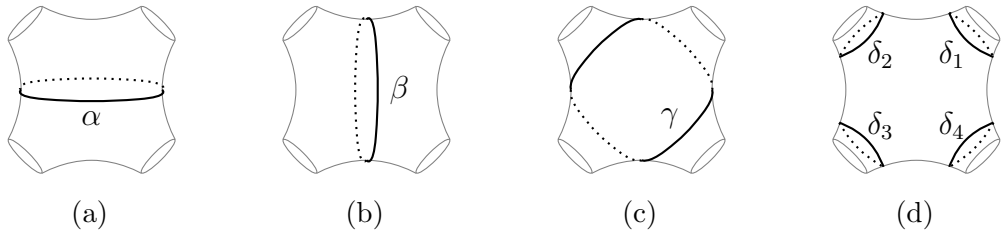

\begin{proposition}[\cite{BP00}, Theorem~3.1]
\label{prop:skeinS04}
The skein algebra $\Sk_A(S_{0,4})$ is generated by $\alpha$, $\beta$, $\gamma$, and the~$\delta_i$, subject to the relations 
\begin{align*}
A^2\alpha\beta-A^{-2}\beta\alpha &= \left(A^4-A^{-4}\right)\gamma+\left(A^2-A^{-2}\right)\left(\delta_2\delta_4+\delta_1\delta_3\right), \\
A^2\beta\gamma-A^{-2}\gamma\beta &= \left(A^4-A^{-4}\right)\alpha+\left(A^2-A^{-2}\right)\left(\delta_1\delta_2+\delta_3\delta_4\right), \\
A^2\gamma\alpha-A^{-2}\alpha\gamma &= \left(A^4-A^{-4}\right)\beta+\left(A^2-A^{-2}\right)\left(\delta_1\delta_4+\delta_2\delta_3\right),
\end{align*}
and the relation 
\begin{align*}
A^2\alpha\beta\gamma = &A^4\alpha^2+A^{-4}\beta^2+A^4\gamma^2 \\
&+A^2\left(\delta_1\delta_2+\delta_3\delta_4\right)\alpha+A^{-2}\left(\delta_1\delta_4+\delta_2\delta_3\right)\beta+A^2\left(\delta_2\delta_4+\delta_1\delta_3\right)\gamma \\
&+\delta_1^2+\delta_2^2+\delta_3^2+\delta_4^2+\delta_1\delta_2\delta_3\delta_4-\left(A^2+A^{-2}\right)^2.
\end{align*}
\end{proposition}

Next we consider the surface $S=S_{1,1}$. Figure~\ref{fig:skeingenerators} shows four elements of~$\Sk_A(S)$, denoted~$\alpha$,~$\beta$,~$\gamma$, and~$\delta$. (Note that we use the same symbols as in the case of the four-holed sphere; in the following, we will always make clear which surface we are referring to.)

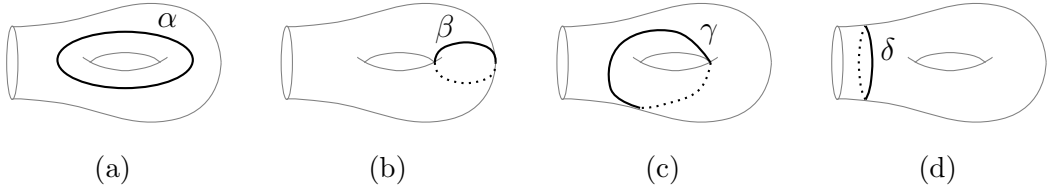
\begin{figure}[ht]
\begin{subfigure}{\textwidth/5}
\begin{center}
\begin{tikzpicture}[scale=0.75]
\draw[gray, thin] (-2,0) ellipse (0.1 and 0.65);
\draw[gray, thin] plot [smooth, tension=1] coordinates { (-2,0.65) (-1.25,0.7) (0.9,1) (1.7,0)};
\draw[gray, thin] plot [smooth, tension=1] coordinates { (-2,-0.65) (-1.25,-0.7) (0.9,-1) (1.7,0)};
\draw[gray, thin] plot [smooth, tension=1] coordinates { (-0.75,0.1) (-0.3,-0.1) (0.3,-0.1) (0.75,0.1)};
\draw[gray, thin] plot [smooth, tension=1] coordinates { (-0.625,0) (-0.3,0.15) (0.3,0.15) (0.625,0)};
\draw[black, thick] (0,0.05) ellipse (1.2 and 0.5);
\node at (0.75,0.75) {$\alpha$};
\end{tikzpicture}
\end{center}
\caption{}
\end{subfigure}
\begin{subfigure}{\textwidth/5}
\begin{center}
\begin{tikzpicture}[scale=0.75]
\draw[gray, thin] (-2,0) ellipse (0.1 and 0.65);
\draw[gray, thin] plot [smooth, tension=1] coordinates { (-2,0.65) (-1.25,0.7) (0.9,1) (1.7,0)};
\draw[gray, thin] plot [smooth, tension=1] coordinates { (-2,-0.65) (-1.25,-0.7) (0.9,-1) (1.7,0)};
\draw[gray, thin] plot [smooth, tension=1] coordinates { (-0.75,0.1) (-0.3,-0.1) (0.3,-0.1) (0.75,0.1)};
\draw[gray, thin] plot [smooth, tension=1] coordinates { (-0.625,0) (-0.3,0.15) (0.3,0.15) (0.625,0)};
\draw[black, thick] plot [smooth, tension=1] coordinates { (0.625,0) (0.85,0.3) (1.5,0.3) (1.7,0)};
\draw[black, thick, dotted] plot [smooth, tension=1] coordinates { (0.625,0) (0.85,-0.3) (1.5,-0.3) (1.7,0)};
\node at (0.8,0.6) {$\beta$};
\end{tikzpicture}
\end{center}
\caption{}
\end{subfigure}
\begin{subfigure}{\textwidth/5}
\begin{center}
\begin{tikzpicture}[scale=0.75]
\draw[gray, thin] (-2,0) ellipse (0.1 and 0.65);
\draw[gray, thin] plot [smooth, tension=1] coordinates { (-2,0.65) (-1.25,0.7) (0.9,1) (1.7,0)};
\draw[gray, thin] plot [smooth, tension=1] coordinates { (-2,-0.65) (-1.25,-0.7) (0.9,-1) (1.7,0)};
\draw[gray, thin] plot [smooth, tension=1] coordinates { (-0.75,0.1) (-0.3,-0.1) (0.3,-0.1) (0.75,0.1)};
\draw[gray, thin] plot [smooth, tension=1] coordinates { (-0.625,0) (-0.3,0.15) (0.3,0.15) (0.625,0)};
\draw[black, thick] plot [smooth, tension=0.75] coordinates { (0.625,0) (0,0.55) (-0.95,0.35) (-1.15,-0.45) (-0.625,-0.8)};
\draw[black, thick, dotted] plot [smooth, tension=0.75] coordinates { (0.625,0) (0.4,-0.5) (-0.25,-0.75) (-0.625,-0.8)};
\node at (0.6,0.5) {$\gamma$};
\end{tikzpicture}
\end{center}
\caption{}
\end{subfigure}
\begin{subfigure}{\textwidth/5}
\begin{center}
\begin{tikzpicture}[scale=0.75]
\draw[gray, thin] (-2,0) ellipse (0.1 and 0.65);
\draw[gray, thin] plot [smooth, tension=1] coordinates { (-2,0.65) (-1.25,0.7) (0.9,1) (1.7,0)};
\draw[gray, thin] plot [smooth, tension=1] coordinates { (-2,-0.65) (-1.25,-0.7) (0.9,-1) (1.7,0)};
\draw[gray, thin] plot [smooth, tension=1] coordinates { (-0.75,0.1) (-0.3,-0.1) (0.3,-0.1) (0.75,0.1)};
\draw[gray, thin] plot [smooth, tension=1] coordinates { (-0.625,0) (-0.3,0.15) (0.3,0.15) (0.625,0)};
\draw[black, thick, dotted] plot [smooth, tension=1] coordinates { (-1.5,-0.65) (-1.6,-0.3) (-1.6,0.3) (-1.5,0.65)};
\draw[black, thick] plot [smooth, tension=1] coordinates { (-1.5,-0.65) (-1.4,-0.3) (-1.4,0.3) (-1.5,0.65)};
\node at (-1.1,0.25) {$\delta$};
\end{tikzpicture}
\end{center}
\caption{}
\end{subfigure}
\caption{Generators of $\Sk_A(S_{1,1})$.\label{fig:skeingenerators}}
\end{figure}

\begin{proposition}[\cite{BP00}, Theorem~2.1]
\label{prop:skeinS11}
The skein algebra $\Sk_A(S_{1,1})$ is generated by $\alpha$, $\beta$, and~$\gamma$, subject to the relations 
\begin{align*}
A^{-1}\alpha\beta-A\beta\alpha &= (A^{-2}-A^2)\gamma, \\
A^{-1}\beta\gamma-A\gamma\beta &= (A^{-2}-A^2)\alpha, \\
A^{-1}\gamma\alpha-A\alpha\gamma &= (A^{-2}-A^2)\beta.
\end{align*}
The element $\delta$ satisfies the further relation 
\[
A^{-2}\alpha^2+A^2\beta^2+A^{-2}\gamma^2-A^{-1}\alpha\beta\gamma=A^2+A^{-2}-\delta.
\]
\end{proposition}

\subsection{Spherical DAHA of type $(C_1^\vee,C_1)$}

Let us write $\mathbb{C}_{q,\mathbf{t}}=\mathbb{C}(q^{\frac{1}{2}},t_1,\dots,t_4)$ and define the \emph{double affine Hecke~algebra~(DAHA)} of type~$(C_1^\vee,C_1)$ to be the $\mathbb{C}_{q,\mathbf{t}}$-algebra $\mathcal{H}_{q,\mathbf{t}}(C_1^\vee,C_1)$ generated by variables $T_1,\dots,T_4$, and their inverses, subject to the relations 
\[
(T_i-t_i)(T_i+t_i^{-1})=0, \quad i=1,\dots,4, \quad \text{and} \quad T_4T_3T_2T_1=q^{-1}.
\]
The first relation implies that the element $e=(t_3+t_3^{-1})^{-1}(T_3+t_3^{-1})$ is an idempotent in~$\mathcal{H}_{q,\mathbf{t}}$, and the \emph{spherical~DAHA} of type $(C_1^\vee,C_1)$ is defined as the algebra $\mathcal{SH}_{q,\mathbf{t}}(C_1^\vee,C_1)=e\mathcal{H}_{q,\mathbf{t}}(C_1^\vee,C_1)e\subset\mathcal{H}_{q,\mathbf{t}}(C_1^\vee,C_1)$. We consider the elements 
\[
x=\left(T_4T_3+(T_4T_3\right)^{-1})e, \quad y=\left(T_3T_2+(T_3T_2)^{-1}\right)e, \quad z=\left(T_3T_1+(T_3T_1)^{-1}\right)e
\]
of $\mathcal{SH}_{q,\mathbf{t}}(C_1^\vee,C_1)$.

\begin{proposition}[\cite{BS18}, Theorem~2.7]
\label{prop:DAHAtypeCC}
The spherical DAHA $\mathcal{SH}_{q,\mathbf{t}}(C_1^\vee,C_1)$ is generated by the elements $x$,~$y$, and~$z$ subject to the relations 
\begin{align*}
q^{-1}xy-qyx &= \left(q^{-2}-q^2\right)z-\left(q^{-1}-q\right)L, \\
q^{-1}yz-qzy &= \left(q^{-2}-q^2\right)x-\left(q^{-1}-q\right)M, \\
q^{-1}zx-qxz &= \left(q^{-2}-q^2\right)y-\left(q^{-1}-q\right)N,
\end{align*}
and 
\begin{align*}
q^{-1}xyz = &q^{-2}x^2+q^2y^2+q^{-2}z^2-q^{-1}Mx-qNy-q^{-1}Lz \\
&-s_1^2-s_2^2-s_3^2-s_4^2+s_1s_2s_3s_4-\left(q+q^{-1}\right)^2
\end{align*}
where $s_1=t_1-t_1^{-1}$, $s_2=t_2-t_2^{-1}$, $s_3=q^{-1}t_3-qt_3^{-1}$, $s_4=t_4-t_4^{-1}$, and $L=s_2s_4+s_3s_1$, $M=s_1s_2+s_3s_4$, $N=s_1s_4+s_3s_2$.
\end{proposition}

Let us recall the polynomial representation of the spherical DAHA. We consider the $\mathbb{C}_{q,\mathbf{t}}$-linear operators $\sigma$,~$\tau\in\End\mathbb{C}_{q,\mathbf{t}}[X^{\pm1}]$ given by the formulas $(\sigma f)(X)=f(X^{-1})$ and $(\tau f)(X)=f(q^2X)$, and we define operators 
\begin{align*}
\widehat{T}_3 &= t_3+t_3^{-1}\frac{\left(1-t_3t_4X\right)\left(1+\frac{t_3}{t_4}X\right)}{1-X^2}(\sigma-1), \\
\widehat{T}_2 &= t_2+t_2^{-1}\frac{\left(1-qt_1t_2X^{-1}\right)\left(1+\frac{qt_2}{t_1}X^{-1}\right)}{1-q^2X^{-2}}(\sigma\tau-1),
\end{align*}
and $\widehat{T}_4=X^{-1}\widehat{T}_3^{-1}$, \ $\widehat{T}_1=q^{-1}\widehat{T}_2^{-1}X$. As explained in Section~2.22 of~\cite{NS04}, the rule $T_i\mapsto\widehat{T}_i$ gives a $\mathbb{C}_{q,\mathbf{t}}$-linear embedding of the DAHA $\mathcal{H}_{q,\mathbf{t}}(C_1^\vee,C_1)$ into $\End\mathbb{C}_{q,\mathbf{t}}[X^{\pm1}]$. There is a $\mathbb{Z}_2$-action on the ring $\mathbb{C}_{q,\mathbf{t}}[X^{\pm1}]$ where the nontrivial element acts by $X\mapsto X^{-1}$ so that the corresponding ring of invariants is $\mathbb{C}_{q,\mathbf{t}}[X^{\pm1}]^{\mathbb{Z}_2}\cong\mathbb{C}_{q,\mathbf{t}}[X+X^{-1}]$.

\begin{proposition}
\label{prop:polyrepCC}
There is a $\mathbb{C}_{q,\mathbf{t}}$-algebra embedding $\mathcal{SH}_{q,\mathbf{t}}(C_1^\vee,C_1)\rightarrow\End\mathbb{C}_{q,\mathbf{t}}[X^{\pm1}]^{\mathbb{Z}_2}$ that maps 
\begin{equation}
\label{eqn:polyrepCC}
\begin{split}
x &\mapsto X+X^{-1}, \\
y &\mapsto U(X)(\tau-1)+U(X^{-1})(\tau^{-1}-1)+f_y(X), \\
z &\mapsto qXU(X)(\tau-1)+qX^{-1}U(X^{-1})(\tau^{-1}-1)+f_z(X),
\end{split}
\end{equation}
where 
\[
U(X)=t_2^{-1}t_3^{-1}\frac{\left(1-t_3t_4X\right)\left(1+\frac{t_3}{t_4}X\right)\left(1-qt_1t_2X\right)\left(1+\frac{qt_2}{t_1}X\right)}{\left(1-X^2\right)\left(1-q^2X^2\right)}
\]
and $f_y(X)$,~$f_z(X)\in\mathbb{C}_{q,\mathbf{t}}[X^{\pm1}]^{\mathbb{Z}_2}$.
\end{proposition}

\begin{proof}
By restricting the polynomial representation defined above, one obtains an injective $\mathbb{C}_{q,\mathbf{t}}$-algebra homomorphism $\mathcal{SH}_{q,\mathbf{t}}(C_1^\vee,C_1)\rightarrow\End\mathbb{C}_{q,\mathbf{t}}[X^{\pm1}]^{\mathbb{Z}_2}$. Let $\widehat{x}$, $\widehat{y}$, $\widehat{z}$ denote the images of $x$, $y$, $z$, respectively, under this homomorphism. From the definitions of~$x$ and~$\widehat{T}_4$, one sees that $\widehat{x}=X+X^{-1}$. The formula for~$\widehat{y}$ is proved in Section~5.8 of~\cite{NS04}. The operator $\widehat{z}$ lies in the $\mathbb{C}_{q,\mathbf{t}}(X)$-vector space spanned by 1, $\tau$, and~$\tau^{-1}$ and can therefore be written 
\[
\widehat{z}=C_1(\tau-1)+C_2(\tau^{-1}-1)+C_3
\]
for unique coefficients $C_i\in\mathbb{C}_{q,\mathbf{t}}(X)$. Since $\widehat{z}$ preserves the space of symmetric Laurent polynomials, we have $f_z(X)\coloneqq C_3=\widehat{z}(1)\in\mathbb{C}_{q,\mathbf{t}}[X^{\pm1}]^{\mathbb{Z}_2}$. The first relation in~Proposition~\ref{prop:DAHAtypeCC} implies that $\widehat{z}=(q^{-2}-q^2)^{-1}(q^{-1}\widehat{x}\widehat{y}-q\widehat{y}\widehat{x})+C_4$ for some $C_4\in\mathbb{C}_{q,\mathbf{t}}(X)$. By computing the coefficients of $\tau$ and~$\tau^{-1}$ on the right hand side of this last equation, one finds that $C_1=qXU(X)$ and $C_2=qX^{-1}U(X^{-1})$. This completes the proof.
\end{proof}

\subsection{A representation of $\Sk_{A,\boldsymbol{\lambda}}(S_{0,4})$}

We will now use the spherical DAHA of type $(C_1^\vee,C_1)$ to define a representation of the relative skein algebra of a four-holed sphere. We use the following~fact, versions of which can be found in~\cite{BS18,CS21,H19}.

\begin{proposition}
\label{prop:skein04DAHACCiso}
There is a $\mathbb{C}$-algebra embedding $\Sk_{A,\boldsymbol{\lambda}}(S_{0,4})\rightarrow\mathcal{SH}_{q,\mathbf{t}}(C_1^\vee,C_1)$ that maps $\alpha\mapsto x$, $\beta\mapsto y$, $\gamma\mapsto z$, $\lambda_j\mapsto \mathrm{i}t_j$ for~$j\neq3$, $\lambda_3\mapsto\mathrm{i}q^{-1}t_3$, and $A\mapsto q^{-\frac{1}{2}}$.
\end{proposition}

\begin{proof}
This follows from the presentations of these algebras given in Propositions~\ref{prop:skeinS04} and~\ref{prop:DAHAtypeCC}.
\end{proof}

From this we obtain the desired representation of the relative skein algebra $\Sk_{A,\boldsymbol{\lambda}}(S_{0,4})$.

\begin{proposition}
\label{prop:polyrepS04}
There is a $\mathbb{C}$-algebra embedding $\Sk_{A,\boldsymbol{\lambda}}(S_{0,4})\rightarrow\End\mathbb{C}_{q,\mathbf{t}}[X^{\pm1}]^{\mathbb{Z}_2}$ that maps $A\mapsto q^{-\frac{1}{2}}$, $\lambda_1\mapsto-t_1$, $\lambda_2\mapsto t_2$, $\lambda_3\mapsto t_3$, $\lambda_4\mapsto-t_4$, and 
\begin{equation}
\label{eqn:polyrepS04}
\begin{split}
\alpha &\mapsto X+X^{-1}, \\
\beta &\mapsto U(X)(\tau-1)+U(X^{-1})(\tau^{-1}-1)+f_y(X), \\
\gamma &\mapsto qXU(X)(\tau-1)+qX^{-1}U(X^{-1})(\tau^{-1}-1)+f_z(X),
\end{split}
\end{equation}
where 
\[
U(X)=t_2^{-1}t_3^{-1}\frac{\left(1-qt_3t_4X\right)\left(1-q\frac{t_3}{t_4}X\right)\left(1-qt_1t_2X\right)\left(1-q\frac{t_2}{t_1}X\right)}{\left(1-X^2\right)\left(1-q^2X^2\right)}
\]
and $f_y(X)$,~$f_z(X)\in\mathbb{C}_{q,\mathbf{t}}[X^{\pm1}]^{\mathbb{Z}_2}$.
\end{proposition}

\begin{proof}
By composing the maps in Propositions~\ref{prop:polyrepCC} and~\ref{prop:skein04DAHACCiso}, we obtain an injective $\mathbb{C}$-algebra homomorphism $\Sk_{A,\boldsymbol{\lambda}}(S_{0,4})\rightarrow\End\mathbb{C}_{q,\mathbf{t}}[X^{\pm1}]^{\mathbb{Z}_2}$ mapping $A\mapsto q^{-\frac{1}{2}}$, $\lambda_1\mapsto\mathrm{i}t_1$, $\lambda_2\mapsto\mathrm{i}t_2$, $\lambda_3\mapsto\mathrm{i}q^{-1}t_3$, $\lambda_4\mapsto\mathrm{i}t_4$, and mapping the generators $\alpha$, $\beta$, $\gamma$ to the operators on the right hand side of~\eqref{eqn:polyrepCC}. Rescaling the scalars by $t_1\mapsto\mathrm{i}t_1$, $t_2\mapsto-\mathrm{i}t_2$, $t_3\mapsto-\mathrm{i}qt_3$, $t_4\mapsto\mathrm{i}t_4$, we obtain a map with the required values on generators.
\end{proof}

Using the formulas~\eqref{eqn:polyrepS04}, we can compute the operators corresponding to other curves on~$S_{0,4}$. Namely, we consider the curves $\gamma_n$~($n\in\mathbb{Z}$) illustrated in Figure~\ref{fig:S04curves}. Note that $\gamma_0$ and~$\gamma_1$ are the curves denoted~$\beta$ and~$\gamma$ above.

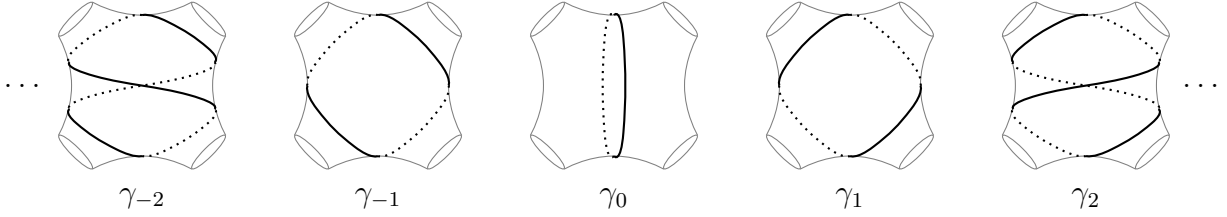
\begin{figure}[ht]
\begin{tikzpicture}[scale=1.25]
\clip(-1.45,-1.4) rectangle (1.2,1.25);
\draw[gray, thin, rotate=45] (1,0) ellipse (0.05 and 0.25);
\draw[gray, thin, rotate=135] (1,0) ellipse (0.05 and 0.25);
\draw[gray, thin, rotate=-135] (1,0) ellipse (0.05 and 0.25);
\draw[gray, thin, rotate=-45] (1,0) ellipse (0.05 and 0.25);
\draw[gray, thin] plot [smooth, tension=1] coordinates { (0.53,0.89) (0,0.75) (-0.53,0.89)};
\draw[gray, thin] plot [smooth, tension=1] coordinates { (-0.89,0.53) (-0.75,0) (-0.89,-0.53)};
\draw[gray, thin] plot [smooth, tension=1] coordinates { (-0.53,-0.89) (0,-0.75) (0.53,-0.89)};
\draw[gray, thin] plot [smooth, tension=1] coordinates { (0.89,-0.53) (0.75,0) (0.89,0.53)};
\draw[black, thick, dotted] plot [smooth, tension=1] coordinates { (0,0.75) (-0.27,0.68) (-0.7,0.4) (-0.77,0.25)};
\draw[black, thick] plot [smooth, tension=1] coordinates { (0,0.75) (0.27,0.68) (0.7,0.4) (0.77,0.25)};
\draw[black, thick, dotted] plot [smooth, tension=1] coordinates { (0,-0.75) (0.27,-0.68) (0.7,-0.4) (0.77,-0.25)};
\draw[black, thick] plot [smooth, tension=1] coordinates { (0,-0.75) (-0.27,-0.68) (-0.7,-0.4) (-0.77,-0.25)};
\draw[black, thick] plot [smooth, tension=1] coordinates { (-0.77,0.25) (-0.5,0.1) (0.5,-0.1) (0.77,-0.25)}; 
\draw[black, thick, dotted] plot [smooth, tension=1] coordinates { (0.77,0.25) (0.5,0.1) (-0.5,-0.1) (-0.77,-0.25)}; 
\node at (0,-1.2) {$\gamma_{-2}$};
\node at (-1.25,0) {$\dots$};
\end{tikzpicture}
\begin{tikzpicture}[scale=1.25]
\clip(-1.2,-1.4) rectangle (1.2,1.25);
\draw[gray, thin, rotate=45] (1,0) ellipse (0.05 and 0.25);
\draw[gray, thin, rotate=135] (1,0) ellipse (0.05 and 0.25);
\draw[gray, thin, rotate=-135] (1,0) ellipse (0.05 and 0.25);
\draw[gray, thin, rotate=-45] (1,0) ellipse (0.05 and 0.25);
\draw[gray, thin] plot [smooth, tension=1] coordinates { (0.53,0.89) (0,0.75) (-0.53,0.89)};
\draw[gray, thin] plot [smooth, tension=1] coordinates { (-0.89,0.53) (-0.75,0) (-0.89,-0.53)};
\draw[gray, thin] plot [smooth, tension=1] coordinates { (-0.53,-0.89) (0,-0.75) (0.53,-0.89)};
\draw[gray, thin] plot [smooth, tension=1] coordinates { (0.89,-0.53) (0.75,0) (0.89,0.53)};
\draw[black, thick, dotted] plot [smooth, tension=1] coordinates { (0,0.75) (-0.25,0.65) (-0.65,0.25) (-0.75,0)};
\draw[black, thick, dotted] plot [smooth, tension=1] coordinates { (0.75,0) (0.65,-0.25) (0.25,-0.65) (0,-0.75)};
\draw[black, thick] plot [smooth, tension=1] coordinates { (0.75,0) (0.65,0.25) (0.25,0.65) (0,0.75)};
\draw[black, thick] plot [smooth, tension=1] coordinates { (-0.75,0) (-0.65,-0.25) (-0.25,-0.65) (0,-0.75)};
\node at (0,-1.2) {$\gamma_{-1}$};
\end{tikzpicture}
\begin{tikzpicture}[scale=1.25]
\clip(-1.2,-1.4) rectangle (1.2,1.25);
\draw[gray, thin, rotate=45] (1,0) ellipse (0.05 and 0.25);
\draw[gray, thin, rotate=135] (1,0) ellipse (0.05 and 0.25);
\draw[gray, thin, rotate=-135] (1,0) ellipse (0.05 and 0.25);
\draw[gray, thin, rotate=-45] (1,0) ellipse (0.05 and 0.25);
\draw[gray, thin] plot [smooth, tension=1] coordinates { (0.53,0.89) (0,0.75) (-0.53,0.89)};
\draw[gray, thin] plot [smooth, tension=1] coordinates { (-0.89,0.53) (-0.75,0) (-0.89,-0.53)};
\draw[gray, thin] plot [smooth, tension=1] coordinates { (-0.53,-0.89) (0,-0.75) (0.53,-0.89)};
\draw[gray, thin] plot [smooth, tension=1] coordinates { (0.89,-0.53) (0.75,0) (0.89,0.53)};
\draw[black, thick, dotted] plot [smooth, tension=1] coordinates { (0,-0.75) (-0.1,-0.5) (-0.1,0.5) (0,0.75)};
\draw[black, thick] plot [smooth, tension=1] coordinates { (0,-0.75) (0.1,-0.5) (0.1,0.5) (0,0.75)};
\node at (0,-1.2) {$\gamma_0$};
\end{tikzpicture}
\begin{tikzpicture}[scale=1.25]
\clip(-1.2,-1.4) rectangle (1.2,1.25);
\draw[gray, thin, rotate=45] (1,0) ellipse (0.05 and 0.25);
\draw[gray, thin, rotate=135] (1,0) ellipse (0.05 and 0.25);
\draw[gray, thin, rotate=-135] (1,0) ellipse (0.05 and 0.25);
\draw[gray, thin, rotate=-45] (1,0) ellipse (0.05 and 0.25);
\draw[gray, thin] plot [smooth, tension=1] coordinates { (0.53,0.89) (0,0.75) (-0.53,0.89)};
\draw[gray, thin] plot [smooth, tension=1] coordinates { (-0.89,0.53) (-0.75,0) (-0.89,-0.53)};
\draw[gray, thin] plot [smooth, tension=1] coordinates { (-0.53,-0.89) (0,-0.75) (0.53,-0.89)};
\draw[gray, thin] plot [smooth, tension=1] coordinates { (0.89,-0.53) (0.75,0) (0.89,0.53)};
\draw[black, thick] plot [smooth, tension=1] coordinates { (0,0.75) (-0.25,0.65) (-0.65,0.25) (-0.75,0)};
\draw[black, thick] plot [smooth, tension=1] coordinates { (0.75,0) (0.65,-0.25) (0.25,-0.65) (0,-0.75)};
\draw[black, thick, dotted] plot [smooth, tension=1] coordinates { (0.75,0) (0.65,0.25) (0.25,0.65) (0,0.75)};
\draw[black, thick, dotted] plot [smooth, tension=1] coordinates { (-0.75,0) (-0.65,-0.25) (-0.25,-0.65) (0,-0.75)};
\node at (0,-1.2) {$\gamma_1$};
\end{tikzpicture}
\begin{tikzpicture}[scale=1.25]
\clip(-1.2,-1.4) rectangle (1.45,1.25);
\draw[gray, thin, rotate=45] (1,0) ellipse (0.05 and 0.25);
\draw[gray, thin, rotate=135] (1,0) ellipse (0.05 and 0.25);
\draw[gray, thin, rotate=-135] (1,0) ellipse (0.05 and 0.25);
\draw[gray, thin, rotate=-45] (1,0) ellipse (0.05 and 0.25);
\draw[gray, thin] plot [smooth, tension=1] coordinates { (0.53,0.89) (0,0.75) (-0.53,0.89)};
\draw[gray, thin] plot [smooth, tension=1] coordinates { (-0.89,0.53) (-0.75,0) (-0.89,-0.53)};
\draw[gray, thin] plot [smooth, tension=1] coordinates { (-0.53,-0.89) (0,-0.75) (0.53,-0.89)};
\draw[gray, thin] plot [smooth, tension=1] coordinates { (0.89,-0.53) (0.75,0) (0.89,0.53)};
\draw[black, thick] plot [smooth, tension=1] coordinates { (0,0.75) (-0.27,0.68) (-0.7,0.4) (-0.77,0.25)};
\draw[black, thick, dotted] plot [smooth, tension=1] coordinates { (0,0.75) (0.27,0.68) (0.7,0.4) (0.77,0.25)};
\draw[black, thick] plot [smooth, tension=1] coordinates { (0,-0.75) (0.27,-0.68) (0.7,-0.4) (0.77,-0.25)};
\draw[black, thick, dotted] plot [smooth, tension=1] coordinates { (0,-0.75) (-0.27,-0.68) (-0.7,-0.4) (-0.77,-0.25)};
\draw[black, thick, dotted] plot [smooth, tension=1] coordinates { (-0.77,0.25) (-0.5,0.1) (0.5,-0.1) (0.77,-0.25)}; 
\draw[black, thick] plot [smooth, tension=1] coordinates { (0.77,0.25) (0.5,0.1) (-0.5,-0.1) (-0.77,-0.25)}; 
\node at (0,-1.2) {$\gamma_2$};
\node at (1.25,0) {$\dots$};
\end{tikzpicture}
\caption{A family of curves on $S_{0,4}$.\label{fig:S04curves}}
\end{figure}

\begin{lemma}
\label{lem:imageS04curves}
The representation of Proposition~\ref{prop:polyrepS04} maps 
\[
\gamma_n\mapsto q^nX^nU(X)(\tau-1)+q^nX^{-n}U(X^{-1})(\tau^{-1}-1)+f_n(X)
\]
where $U$ is the operator defined there and $f_n(X)\in\mathbb{C}_{q,\mathbf{t}}[X^{\pm1}]^{\mathbb{Z}_2}$.
\end{lemma}

\begin{proof}
For $n\in\{0,1\}$ the statement follows from Proposition~\ref{prop:polyrepS04} where we define $f_0(X)\coloneqq f_y(X)$ and $f_1(X)\coloneqq f_z(X)$. Let us fix an integer $m\geq1$ and assume inductively that the statement is true for all nonnegative integers $n\leq m$. One can check that the skein relations depicted in Figures~\ref{subfig:resolution} and~\ref{fig:puncturerelation} imply $\alpha\gamma_m=A^2\gamma_{m+1}+A^{-2}\gamma_{m-1}+C_m$ for some $C_m\in\mathbb{C}[A^{\pm1},\lambda_1^{\pm1},\dots,\lambda_4^{\pm1}]$. It follows from our assumption that 
\begin{align*}
\gamma_{m+1} &= A^{-2}\alpha\gamma_m-A^{-4}\gamma_{m-1}-A^{-2}C_m \\
&\mapsto q^{m+1}X^{m+1}U(X)(\tau-1)+q^{m+1}X^{-m-1}U(X^{-1})(\tau^{-1}-1)+f_{m+1}(X)
\end{align*}
where $f_{m+1}(X)\coloneqq q(X+X^{-1})f_m(X)-q^2f_{m-1}(X)-qC_m'$ with $C_m'$ being the image of~$C_m$ under the homomorphism. By induction, the desired statement is true for all integers $n\geq0$. A similar inductive argument proves the statement for all integers $n\leq0$.
\end{proof}

\subsection{Spherical DAHA of type $A_1$}

We will also write $\mathbb{C}_{q,t}=\mathbb{C}(q^{\frac{1}{2}},t^{\frac{1}{2}})$ and define the \emph{double affine Hecke algebra (DAHA)} of type~$A_1$ to be the $\mathbb{C}_{q,t}$-algebra $\mathcal{H}_{q,t}(A_1)$ generated by variables $T$, $X^{\pm1}$, and $Y^{\pm1}$, subject to the relations 
\[
TXT=X^{-1}, \quad TY^{-1}T=Y, \quad Y^{-1}X^{-1}YXT^2=q^{-1}, \quad (T-t^{\frac{1}{2}})(T+t^{-\frac{1}{2}})=0.
\]
The last relation implies that the element $e=(t^{\frac{1}{2}}+t^{-\frac{1}{2}})^{-1}(T+t^{-\frac{1}{2}})$ is an idempotent in~$\mathcal{H}_{q,t}(A_1)$. The \emph{spherical~DAHA} is defined as the algebra $\mathcal{SH}_{q,t}(A_1)=e\mathcal{H}_{q,t}(A_1)e\subset\mathcal{H}_{q,t}(A_1)$. We consider the elements 
\[
x=(X+X^{-1})e, \quad y=(Y+Y^{-1})e, \quad z=(q^{\frac{1}{2}}YX+q^{-\frac{1}{2}}X^{-1}Y^{-1})e
\]
of $\mathcal{SH}_{q,t}(A_1)$.

\begin{proposition}[\cite{S19}, Theorem~2.27]
\label{prop:DAHAtypeA}
The spherical DAHA $\mathcal{SH}_{q,t}(A_1)$ is generated by the elements $x$,~$y$, and~$z$ subject to the relations 
\begin{align*}
q^{\frac{1}{2}}xy-q^{-\frac{1}{2}}yx &= (q-q^{-1})z, \\
q^{\frac{1}{2}}yz-q^{-\frac{1}{2}}zy &= (q-q^{-1})x, \\
q^{\frac{1}{2}}zx-q^{-\frac{1}{2}}xz &= (q-q^{-1})y,
\end{align*}
and the relation 
\[
qx^2+q^{-1}y^2+qz^2-q^{\frac{1}{2}}xyz=tq^{-1}-qt^{-1}+q+q^{-1}.
\]
\end{proposition}

To define the polynomial representation of the spherical DAHA, we consider the $\mathbb{C}_{q,t}$-linear operators $\sigma$,~$\varpi\in\End\mathbb{C}_{q,t}[X^{\pm1}]$ given by the formulas $(\sigma f)(X)=f(X^{-1})$ and $(\varpi f)(X)=f(qX)$, and we define operators 
\[
\widehat{T}=t^{\frac{1}{2}}\sigma+\frac{t^{\frac{1}{2}}-t^{-\frac{1}{2}}}{X^2-1}, \quad \widehat{X}=X, \quad \widehat{Y}=\sigma\varpi T.
\]
By Theorem~2.5.6 of~\cite{C05}, the rules $T\mapsto\widehat{T}$, $X\mapsto\widehat{X}$, $Y\mapsto\widehat{Y}$ define a $\mathbb{C}_{q,t}$-linear embedding of the DAHA $\mathcal{H}_{q,t}(A_1)$ into $\End\mathbb{C}_{q,t}[X^{\pm1}]$. As before there is a $\mathbb{Z}_2$-action on the ring $\mathbb{C}_{q,t}[X^{\pm1}]$ where the nontrivial element acts by $X\mapsto X^{-1}$ so that the corresponding ring of invariants is $\mathbb{C}_{q,t}[X^{\pm1}]^{\mathbb{Z}_2}\cong\mathbb{C}_{q,t}[X+X^{-1}]$.

\begin{proposition}
\label{prop:polyrepA}
There is a $\mathbb{C}_{q,t}$-algebra embedding $\mathcal{SH}_{q,t}(A_1)\rightarrow\End\mathbb{C}_{q,t}[X^{\pm1}]^{\mathbb{Z}_2}$ that maps 
\begin{align*}
x &\mapsto X+X^{-1}, \\
y &\mapsto V(X)\varpi+V(X^{-1})\varpi^{-1}, \\
z &\mapsto q^{-\frac{1}{2}}X^{-1}V(X)\varpi+q^{-\frac{1}{2}}XV(X^{-1})\varpi^{-1},
\end{align*}
where 
\[
V(X)=\frac{t^{\frac{1}{2}}X-t^{-\frac{1}{2}}X^{-1}}{X-X^{-1}}.
\]
\end{proposition}

\begin{proof}
By restricting the polynomial representation defined above, one obtains an injective $\mathbb{C}_{q,t}$-algebra homomorphism $\mathcal{SH}_{q,t}(A_1)\rightarrow\End\mathbb{C}_{q,t}[X^{\pm1}]^{\mathbb{Z}_2}$. It follows immediately from the definition of this polynomial representation that~$x$ corresponds to the operator~$X+X^{-1}$. The formula for the operator corresponding to~$y$ is given, for example, in equation~(2.19) of~\cite{S19}. Finally, the~formula for the operator corresponding to~$z$ follows by a short calculation using the first relation in~Proposition~\ref{prop:DAHAtypeA}.
\end{proof}

\subsection{A representation of $\Sk_{A,\lambda}(S_{1,1})$}

Finally, we will use the spherical DAHA of type~$A_1$ to define a representation of the relative skein algebra of a one-holed torus. We use the following fact, versions of which can be found in~\cite{CS21,H19,S19}.

\begin{proposition}
\label{prop:skeinS11DAHAAiso}
There is a $\mathbb{C}$-algebra embedding $\Sk_{A,\rho}(S_{1,1})\rightarrow\mathcal{SH}_{q,t}(A_1)$ that maps $\alpha\mapsto x$, $\beta\mapsto y$, $\gamma\mapsto z$, $\rho\mapsto q^{-\frac{1}{2}}t^{\frac{1}{2}}$, and $A\mapsto q^{-\frac{1}{2}}$. 
\end{proposition}

\begin{proof}
This follows from the presentations of these algebras given in Propositions~\ref{prop:skeinS11} and~\ref{prop:DAHAtypeA}.
\end{proof}

From this we obtain the desired representation of~$\Sk_{A,\lambda}(S_{1,1})$.

\begin{proposition}
\label{prop:polyrepS11}
There is a $\mathbb{C}$-algebra embedding $\Sk_{A,\lambda}(S_{1,1})\rightarrow\End\mathbb{C}_{q,t}[X^{\pm1}]^{\mathbb{Z}_2}$ that maps $A\mapsto q^{-\frac{1}{2}}$, $\lambda\mapsto q^{-1}t$, and 
\begin{equation}
\label{eqn:polyrepS11}
\begin{split}
\alpha &\mapsto X+X^{-1}, \\
\beta &\mapsto V(X)\varpi+V(X^{-1})\varpi^{-1}, \\
\gamma &\mapsto q^{-\frac{1}{2}}X^{-1}V(X)\varpi+q^{-\frac{1}{2}}XV(X^{-1})\varpi^{-1},
\end{split}
\end{equation}
where $V$ is the operator defined in Proposition~\ref{prop:polyrepA}.
\end{proposition}

\begin{proof}
By composing the maps in Propositions~\ref{prop:skeinS11DAHAAiso} and~\ref{prop:polyrepA}, we obtain an injective $\mathbb{C}$-algebra homomorphism $\Sk_{A,\rho}(S_{1,1})\rightarrow\End\mathbb{C}_{q,t}[X^{\pm1}]^{\mathbb{Z}_2}$. Restricting scalars, we obtain an injective $\mathbb{C}$-algebra homomorphism $\Sk_{A,\lambda}(S_{1,1})\rightarrow\End\mathbb{C}_{q,t}[X^{\pm1}]^{\mathbb{Z}_2}$ with the required values on generators.
\end{proof}

Using the formulas~\eqref{eqn:polyrepS11}, we can compute the operators corresponding to other curves on~$S_{1,1}$. Namely, we consider the curves $\gamma_n$~($n\in\mathbb{Z}$) illustrated in Figure~\ref{fig:S04curves}. Note that $\gamma_0$ and~$\gamma_1$ are the curves denoted~$\beta$ and~$\gamma$ above.

\begin{figure}[ht]
\begin{tikzpicture}[scale=0.75]
\clip(-2.95,-2) rectangle (1.9,1.5);
\draw[gray, thin] (-2,0) ellipse (0.1 and 0.65);
\draw[gray, thin] plot [smooth, tension=1] coordinates { (-2,0.65) (-1.25,0.7) (0.9,1) (1.7,0)};
\draw[gray, thin] plot [smooth, tension=1] coordinates { (-2,-0.65) (-1.25,-0.7) (0.9,-1) (1.7,0)};
\draw[gray, thin] plot [smooth, tension=1] coordinates { (-0.75,0.1) (-0.3,-0.1) (0.3,-0.1) (0.75,0.1)};
\draw[gray, thin] plot [smooth, tension=1] coordinates { (-0.625,0) (-0.3,0.15) (0.3,0.15) (0.625,0)};
\draw[black, thick] plot [smooth, tension=0.6] coordinates { (-0.625,0) (-0.6,0.3) (0,0.45) (0.9,0.25) (1.05,-0.15) (0.65,-0.4) (-0.1,-0.43) (-0.8,-0.2) (-1,0.15) (-0.8,0.5) (-0.1,0.7) (0.5,0.7) (1,0.55) (1.35,0.15) (1.3,-0.35) (1,-0.77) (0.625,-1.05)};
\draw[black, thick, dotted] plot [smooth, tension=0.75] coordinates { (-0.625,0) (-0.4,-0.5) (0.2,-0.9) (0.625,-1.05)};
\node at (0,-1.5) {$\gamma_{-2}$};
\node at (-2.6,0) {$\dots$};
\end{tikzpicture}
\begin{tikzpicture}[scale=0.75]
\clip(-2.25,-2) rectangle (1.9,1.5);
\draw[gray, thin] (-2,0) ellipse (0.1 and 0.65);
\draw[gray, thin] plot [smooth, tension=1] coordinates { (-2,0.65) (-1.25,0.7) (0.9,1) (1.7,0)};
\draw[gray, thin] plot [smooth, tension=1] coordinates { (-2,-0.65) (-1.25,-0.7) (0.9,-1) (1.7,0)};
\draw[gray, thin] plot [smooth, tension=1] coordinates { (-0.75,0.1) (-0.3,-0.1) (0.3,-0.1) (0.75,0.1)};
\draw[gray, thin] plot [smooth, tension=1] coordinates { (-0.625,0) (-0.3,0.15) (0.3,0.15) (0.625,0)};
\draw[black, thick] plot [smooth, tension=0.75] coordinates { (-0.625,0) (0,0.55) (0.95,0.35) (1.15,-0.45) (0.625,-1.05)};
\draw[black, thick, dotted] plot [smooth, tension=0.75] coordinates { (-0.625,0) (-0.4,-0.5) (0.2,-0.9) (0.625,-1.05)};
\node at (0,-1.5) {$\gamma_{-1}$};
\end{tikzpicture}
\begin{tikzpicture}[scale=0.75]
\clip(-2.25,-2) rectangle (1.9,1.5);
\draw[gray, thin] (-2,0) ellipse (0.1 and 0.65);
\draw[gray, thin] plot [smooth, tension=1] coordinates { (-2,0.65) (-1.25,0.7) (0.9,1) (1.7,0)};
\draw[gray, thin] plot [smooth, tension=1] coordinates { (-2,-0.65) (-1.25,-0.7) (0.9,-1) (1.7,0)};
\draw[gray, thin] plot [smooth, tension=1] coordinates { (-0.75,0.1) (-0.3,-0.1) (0.3,-0.1) (0.75,0.1)};
\draw[gray, thin] plot [smooth, tension=1] coordinates { (-0.625,0) (-0.3,0.15) (0.3,0.15) (0.625,0)};
\draw[black, thick] plot [smooth, tension=1] coordinates { (0.625,0) (0.85,0.3) (1.5,0.3) (1.7,0)};
\draw[black, thick, dotted] plot [smooth, tension=1] coordinates { (0.625,0) (0.85,-0.3) (1.5,-0.3) (1.7,0)};
\node at (0,-1.5) {$\gamma_0$};
\end{tikzpicture}
\begin{tikzpicture}[scale=0.75]
\clip(-2.25,-2) rectangle (1.9,1.5);
\draw[gray, thin] (-2,0) ellipse (0.1 and 0.65);
\draw[gray, thin] plot [smooth, tension=1] coordinates { (-2,0.65) (-1.25,0.7) (0.9,1) (1.7,0)};
\draw[gray, thin] plot [smooth, tension=1] coordinates { (-2,-0.65) (-1.25,-0.7) (0.9,-1) (1.7,0)};
\draw[gray, thin] plot [smooth, tension=1] coordinates { (-0.75,0.1) (-0.3,-0.1) (0.3,-0.1) (0.75,0.1)};
\draw[gray, thin] plot [smooth, tension=1] coordinates { (-0.625,0) (-0.3,0.15) (0.3,0.15) (0.625,0)};
\draw[black, thick] plot [smooth, tension=0.75] coordinates { (0.625,0) (0,0.55) (-0.95,0.35) (-1.15,-0.45) (-0.625,-0.8)};
\draw[black, thick, dotted] plot [smooth, tension=0.75] coordinates { (0.625,0) (0.4,-0.5) (-0.25,-0.75) (-0.625,-0.8)};
\node at (0,-1.5) {$\gamma_1$};
\end{tikzpicture}
\begin{tikzpicture}[scale=0.75]
\clip(-2.25,-2) rectangle (2.55,1.5);
\draw[gray, thin] (-2,0) ellipse (0.1 and 0.65);
\draw[gray, thin] plot [smooth, tension=1] coordinates { (-2,0.65) (-1.25,0.7) (0.9,1) (1.7,0)};
\draw[gray, thin] plot [smooth, tension=1] coordinates { (-2,-0.65) (-1.25,-0.7) (0.9,-1) (1.7,0)};
\draw[gray, thin] plot [smooth, tension=1] coordinates { (-0.75,0.1) (-0.3,-0.1) (0.3,-0.1) (0.75,0.1)};
\draw[gray, thin] plot [smooth, tension=1] coordinates { (-0.625,0) (-0.3,0.15) (0.3,0.15) (0.625,0)};
\draw[black, thick] plot [smooth, tension=0.6] coordinates { (0.625,0) (0.58,0.3) (0,0.44) (-0.9,0.21) (-1.05,-0.15) (-0.65,-0.4) (0.1,-0.45) (0.8,-0.2) (1,0.2) (0.78,0.58) (0.1,0.7) (-0.5,0.62) (-1,0.42) (-1.35,0.05) (-1.25,-0.45) (-0.625,-0.8)};
\draw[black, thick, dotted] plot [smooth, tension=0.75] coordinates { (0.625,0) (0.4,-0.5) (-0.25,-0.75) (-0.625,-0.8)};
\node at (0,-1.5) {$\gamma_2$};
\node at (2.3,0) {$\dots$};
\end{tikzpicture}
\caption{A family of curves on $S_{1,1}$.\label{fig:S11curves}}
\end{figure}
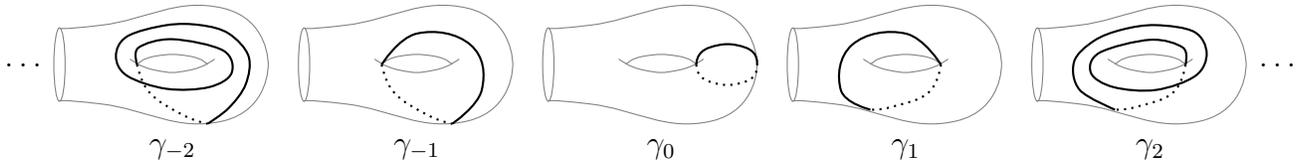

\begin{lemma}
\label{lem:imageS11curves}
The representation of Proposition~\ref{prop:polyrepS04} maps 
\[
\gamma_n\mapsto q^{-\frac{n}{2}}X^{-n}V(X)\varpi+q^{-\frac{n}{2}}X^nV(X^{-1})\varpi^{-1}
\]
where $V$ is the operator defined in Proposition~\ref{prop:polyrepA}.
\end{lemma}

\begin{proof}
For $n\in\{0,1\}$ the statement follows from Proposition~\ref{prop:polyrepS11}. Let us fix an integer $m\geq1$ and assume inductively that the statement is true for all nonnegative integers $n\leq m$. One can check that the skein relation depicted in Figures~\ref{subfig:resolution} implies $\alpha\gamma_m=A\gamma_{m-1}+A^{-1}\gamma_{m+1}$. It follows from our assumption that 
\begin{align*}
\gamma_{m+1} &= A\alpha\gamma_m-A^2\gamma_{m-1} \\
&\mapsto q^{-\frac{m+1}{2}}X^{-m-1}V(X)\varpi+q^{-\frac{m+1}{2}}X^{m+1}V(X^{-1})\varpi^{-1}.
\end{align*}
By induction, the desired statement is true for all integers $n\geq0$. A similar inductive argument proves the statement for all integers $n\leq0$.
\end{proof}

\section{Quantized $K$-theoretic Coulomb branches}
\label{sec:QuantizedKTheoreticCoulombBranches}

In this section, we review the construction of quantized $K$-theoretic Coulomb branches and establish some preliminary results about these objects. For completeness, we prove $K$-theoretic versions of some results that are proved in the literature only for homology.

\subsection{The variety of triples}

Throughout this section, we will write $\mathcal{O}=\mathbb{C}\llbracket z\rrbracket$ for the ring of formal power series and $\mathcal{K}=\mathbb{C}(\!( z)\!)$ for the field of formal Laurent series with complex coefficients. We~will suppose that we are given a complex, connected, reductive algebraic group~$G$, and we will write $G_{\mathcal{O}}$ for the $\mathcal{O}$-valued points and $G_{\mathcal{K}}$ for the $\mathcal{K}$-valued points of~$G$. The \emph{affine Grassmannian} of~$G$ is defined as 
\[
\Gr_G\coloneqq G_{\mathcal{K}}/G_{\mathcal{O}}.
\]
This can be understood as the set of complex points of an ind-scheme defined by an explicit functor of points, and in the following we will work with its reduced part.

In addition to the group~$G$, let us suppose that we are given a complex representation $N$ of~$G$, and let us write $N_\mathcal{O}=N\otimes_{\mathbb{C}}\mathcal{O}$ and $N_\mathcal{K}=N\otimes_{\mathbb{C}}\mathcal{K}$. Following Braverman, Finkelberg, and Nakajima~\cite{BFN18}, we define 
\[
\mathcal{T}_{G,N}\coloneqq G_{\mathcal{K}}\times_{G_\mathcal{O}}N_{\mathcal{O}}
\]
where the right hand side is the quotient of $G_{\mathcal{K}}\times N_{\mathcal{O}}$ by the equivalence relation given by $(g,n)\sim(gb^{-1},bn)$ for every $b\in G_{\mathcal{O}}$. We denote the equivalence class of a pair $(g,n)$ with respect to this relation by~$[g,n]$. Then there is a well defined map $\Pi:\mathcal{T}_{G,N}\rightarrow N_{\mathcal{K}}$ given by $[g,n]\mapsto gn$, and we define the \emph{variety of triples} $\mathcal{R}_{G,N}$ to be 
\[
\mathcal{R}_{G,N}\coloneqq\Pi^{-1}(N_{\mathcal{O}})=\{[g,n]\in\mathcal{T}_{G,N}:gn\in N_{\mathcal{O}}\}.
\]
Like the affine Grassmannian, the spaces $\mathcal{T}_{G,N}$ and $\mathcal{R}_{G,N}$ can be understood as sets of complex points of corresponding ind-schemes, and we take the reduced parts. These ind-schemes parametrize triples of geometric data on a formal disk; see~\cite{BFN18} for details. When there is no possibility of confusion, we will simply write $\mathcal{T}=\mathcal{T}_{G,N}$ and $\mathcal{R}=\mathcal{R}_{G,N}$.

In addition to the map~$\Pi$, there is a well defined projection $\pi:\mathcal{T}\rightarrow\Gr_G$ given by $[g,n]\mapsto gG_{\mathcal{O}}$. The maps $\Pi$ and~$\pi$ provide an injective map of sets 
\[
(\pi,\Pi):\mathcal{T}\hookrightarrow\Gr_G\times N_{\mathcal{K}}.
\]
From the definition of the variety of triples, one sees that 
\[
\mathcal{R}=\mathcal{T}\cap\left(\Gr_G\times N_{\mathcal{O}}\right).
\]
The group $G_{\mathcal{K}}$ acts on~$\mathcal{T}$ by left multiplication, and the subgroup $G_{\mathcal{O}}$ preserves~$\mathcal{R}$.

\subsection{Equivariant $K$-theory}

The $K$-theoretic Coulomb branch is defined in terms of the equivariant $K$-theory of the variety of triples $\mathcal{R}_{G,N}$. Since the space $\mathcal{R}_{G,N}$ is an ind-scheme and not a true variety, we must take special care in defining this $K$-theory. We use the fact that there is a natural $G_\mathcal{O}$-action on $\Gr_G$ and a decomposition 
\[
\Gr_G=\bigsqcup_{\lambda\in X_+}\Gr_G^\lambda
\]
into $G_{\mathcal{O}}$-orbits $\Gr_G^\lambda$ parametrized by dominant coweights $\lambda$ of~$G$. The closure $\overline{\Gr}_G^\lambda$ of such an orbit is a projective variety. Moreover, using the standard partial order on the set~$X_+$ of dominant coweights, one has $\overline{\Gr}_G^\lambda=\bigsqcup_{\mu\leq\lambda}\Gr_G^\mu$. We define $\mathcal{T}_{\leq\lambda}=\pi^{-1}\left(\overline{\Gr}_G^\lambda\right)$, where $\pi:\mathcal{T}\rightarrow\Gr_G$ is the projection, and define $\mathcal{R}_{\leq\lambda}=\mathcal{R}\cap\mathcal{T}_{\leq\lambda}$.

Next we consider further decompositions of the spaces $\mathcal{T}_{\leq\lambda}$ and $\mathcal{R}_{\leq\lambda}$. For any integer $d>0$, we~can form the quotient $\mathcal{T}^d=G_{\mathcal{K}}\times_{G_{\mathcal{O}}}(N_{\mathcal{O}}/z^dN_{\mathcal{O}})$ of~$\mathcal{T}$ by the fiberwise translation by $z^dN_{\mathcal{O}}$. For~$e>d$ there is an induced map $p_d^e:\mathcal{T}^e\rightarrow\mathcal{T}^d$, and we recover~$\mathcal{T}$ as the inverse limit of this system. Let us write $\mathcal{T}_{\leq\lambda}^d$ for the image of~$\mathcal{T}_{\leq\lambda}$ in the quotient~$\mathcal{T}^d$. The $p_d^e$ restrict to maps $p_d^e:\mathcal{T}_{\leq\lambda}^e\rightarrow\mathcal{T}_{\leq\lambda}^d$, and $\mathcal{T}_{\leq\lambda}$ is the inverse limit of this system.

If we choose $d\in\mathbb{Z}$ greater than a certain constant depending on~$\lambda$, then it follows that the fiberwise action of $z^dN_{\mathcal{O}}$ on~$\mathcal{T}$ stabilizes~$\mathcal{R}_{\leq\lambda}$. Hence for $d\gg0$ we can form the quotient $\mathcal{R}_{\leq\lambda}^d$ of~$\mathcal{R}_{\leq\lambda}$ by this action. It is a closed subset of $\mathcal{T}_{\leq\lambda}^d$. For $e>d$ the $p_d^e$ restrict to maps $p_d^e:\mathcal{R}_{\leq\lambda}^e\rightarrow\mathcal{R}_{\leq\lambda}^d$, and $\mathcal{R}_{\leq\lambda}$ is the inverse limit of this system.

We now consider group actions on the variety of triples. We suppose that there is an algebraic torus $F\cong(\mathbb{C}^*)^n$ acting on~$N$, and we consider the product group $\widetilde{G}=G\times F$. In physics terminology, $G$ is called a \emph{gauge group} while $F$ is called a \emph{flavor~symmetry group}. There is an action of  $\widetilde{G}_{\mathcal{O}}$ on~$\mathcal{R}$ extending the action of~$G_{\mathcal{O}}$. We can extend this to an action of $\widetilde{G}_{\mathcal{O}}\rtimes\mathbb{C}^*$ on~$\mathcal{R}$ where the $\mathbb{C}^*$ factor acts by rescaling the variable~$z$. Then there are induced $\widetilde{G}_{\mathcal{O}}$- and $\widetilde{G}_{\mathcal{O}}\rtimes\mathbb{C}^*$-actions on $\mathcal{R}_{\leq\lambda}^d$. They descend to actions of $\widetilde{G}_i$ and $\widetilde{G}_i\rtimes\mathbb{C}^*$ for sufficiently large~$i$ where $\widetilde{G}_i=\widetilde{G}(\mathcal{O}/z^i\mathcal{O})$. The space $\mathcal{R}_{\leq\lambda}^d$ is a variety, and so we can talk about its equivariant $K$-theory, as defined for example in~\cite{CG97}.

\begin{definition}[\cite{BFN18,VV10}]
\label{def:KtheoryR}
The $\widetilde{G}_{\mathcal{O}}$-equivariant $K$-theory of $\mathcal{R}_{G,N}$ is defined as the limit 
\[
K^{\widetilde{G}_\mathcal{O}}\left(\mathcal{R}_{G,N}\right)\coloneqq\lim_{\lambda}K^{\widetilde{G}_\mathcal{O}}\left(\mathcal{R}_{\leq\lambda}\right)
\]
where $K^{\widetilde{G}_\mathcal{O}}\left(\mathcal{R}_{\leq\lambda}\right)\coloneqq K^{\widetilde{G}_i}(\mathcal{R}_{\leq\lambda}^d)$ for some $d$,~$i\gg0$ and the limit diagram consists of all pushforward maps $K^{\widetilde{G}_\mathcal{O}}\left(\mathcal{R}_{\leq\lambda}\right)\rightarrow K^{\widetilde{G}_\mathcal{O}}\left(\mathcal{R}_{\leq\mu}\right)$ induced by embeddings $\mathcal{R}_{\leq\lambda}\hookrightarrow\mathcal{R}_{\leq\mu}$ for $\lambda\leq\mu$. The $\widetilde{G}_{\mathcal{O}}\rtimes\mathbb{C}^*$-equivariant $K$-theory $K^{\widetilde{G}_\mathcal{O}\rtimes\mathbb{C}^*}\left(\mathcal{R}_{G,N}\right)$ is defined similarly.
\end{definition}

The vector space $K^{\widetilde{G}_\mathcal{O}}\left(\mathcal{R}_{G,N}\right)$ in fact has a convolution product $*$, which makes it into a commutative $\mathbb{C}$-algebra by Remark~3.14 of~\cite{BFN18}. The space $K^{\widetilde{G}_\mathcal{O}\rtimes\mathbb{C}^*}\left(\mathcal{R}_{G,N}\right)$ likewise has a convolution product~$*$, which makes it into a noncommutative algebra over $K^{\mathbb{C}^*}(\mathrm{pt})\cong\mathbb{C}[q^{\pm1}]$. In the following, we will view it as an algebra over $\mathbb{C}[q^{\pm\frac{1}{2}}]$ by replacing the group $\mathbb{C}^*$ by its standard double cover. We refer to Remark~3.9(3) of~\cite{BFN18} for a detailed discussion of the convolution products in equivariant $K$-theory.

\begin{definition}[\cite{BFN18}]
\label{def:Coulombbranch}
The \emph{($K$-theoretic) Coulomb branch} of $(\widetilde{G},N)$ is the spectrum 
\[
\mathcal{M}_{\mathrm{C}}(\widetilde{G},N)\coloneqq\Spec\left(K^{\widetilde{G}_{\mathcal{O}}}(\mathcal{R}_{G,N}),*\right)
\]
of the algebra defined by the vector space $K^{\widetilde{G}_{\mathcal{O}}}(\mathcal{R}_{G,N})$ with its convolution product $*$. The \emph{quantized (K-theoretic) Coulomb branch} is defined as the noncommutative algebra $(K^{\widetilde{G}_\mathcal{O}\rtimes\mathbb{C}^*}\left(\mathcal{R}_{G,N}\right),*)$.
\end{definition}

Here we abuse notation slightly since we do not indicate the factorization $\widetilde{G}\cong G\times F$ explicitly in our notation for the Coulomb branch or its quantization.

\subsection{The abelian case}
\label{sec:TheAbelianCase}

Let us now consider the group $G=\mathrm{GL}_2^m$ and its subgroup $T=T_{\mathrm{GL}_2}^m$ where $T_{\mathrm{GL}_2}\subset\mathrm{GL}_2$ is the diagonal torus. We have bijections $\mathcal{R}_{T,0}\cong\Gr_T\cong T_{\mathcal{K}}/T_{\mathcal{O}}$. Note that $\mathcal{K}^*=\mathcal{K}\setminus\{0\}$ while $\mathcal{O}^*$ is the  set of formal power series with nonzero constant term. If~$f=\sum_{i\geq k}a_iz^i$ is an element of $\mathcal{K}^*$ with $a_k\neq0$, then we have $f=z^kg$ where $g=\sum_{i\geq0}a_{k+i}z^i$ is an element of~$\mathcal{O}^*$. It follows that any element of~$\mathcal{R}_{T,0}$ is represented by a matrix of the form 
\[
\begin{pmatrix}
z^{k_1} & & \\ 
& \ddots & \\
& & z^{k_{2m}}
\end{pmatrix}
\in T_{\mathcal{K}}
\]
for some $k_1,\dots,k_{2m}\in\mathbb{Z}$. Hence the space $\mathcal{R}_{T,0}$ is naturally identified with the lattice $X_*(T)$ of cocharacters of~$T$ with the discrete topology.

We would like to compute the $(T\times F)_{\mathcal{O}}\rtimes\mathbb{C}^*$-equivariant $K$-theory of this space where $F$ is an algebraic torus. In fact, we will assume $F=F_1\times F_2$ where $F_1=T_{\mathrm{GL}_2}^n$ and $F_2=(\mathbb{C}^*)^d$. Since $T_{\mathcal{O}}$ acts on $\mathcal{R}_{T,0}=T_{\mathcal{K}}/T_{\mathcal{O}}$ by left multiplication and $T_{\mathcal{K}}$ is abelian, we see that $(T\times F)_{\mathcal{O}}$ acts trivially on $\mathcal{R}_{T,0}$. The action of~$\mathbb{C}^*$ on $\mathcal{R}_{T,0}$ by rescaling~$z$ is also trivial. Therefore $(T\times F)_i\rtimes\mathbb{C}^*$ acts trivially on $\mathcal{R}_{\leq\lambda}^d$ for any $i\geq1$, and we have $K^{(T\times F)_{\mathcal{O}}\rtimes\mathbb{C}^*}(\mathcal{R}_{\leq\lambda})=K^{T\times F\rtimes\mathbb{C}^*}(\mathcal{R}_{\leq\lambda}^d)$. Then by definition 
\[
K^{(T\times F)_{\mathcal{O}}\rtimes\mathbb{C}^*}(\mathcal{R}_{T,0})\cong\bigoplus_{\mu\in X_*(T)}K^{T\times F\rtimes\mathbb{C}^*}(\mathrm{pt})
\]
as vector spaces where $\mathrm{pt}$ is the variety consisting of a single point. Note that the actions of $(T\times F)_{\mathcal{O}}$ and $\mathbb{C}^*$ on~$\mathcal{R}_{T,0}$ commute, and so the semidirect product is actually direct.

For any linear algebraic group $\Gamma$, the $\Gamma$-equivariant $K$-theory of a point is the group algebra $\mathbb{C}[X^*(\Gamma)]$ of the character lattice $X^*(\Gamma)$ (see~\cite{CG97}, Section~5.2.1). We will write 
\[
\mathbb{C}[X^*(T\times F\times\mathbb{C}^*)]\cong\mathbb{C}\left[q^{\pm\frac{1}{2}},z_{k,1}^{\pm1},z_{k,2}^{\pm1},z_a^{\pm1},w_{i,1}^{\pm1},w_{i,2}^{\pm1}\right].
\]
Here the variables $z_{k,1}$ and $z_{k,2}$ correspond to the characters of the $k$th factor of $F_1=T_{\mathrm{GL}_2}^n$ given by $\diag(t_1,t_2)\mapsto t_1$ and $\diag(t_1,t_2)\mapsto t_2$, respectively. The variable $z_a$ corresponds to the character of $F_2=(\mathbb{C}^*)^d$ that projects onto the $a$th factor. Finally, the variables $w_{i,1}$ and $w_{i,2}$ correspond to the characters of the $i$th factor~$T_{\mathrm{GL}_2}\subset T$ given by $\diag(t_1,t_2)\mapsto t_1$ and $\diag(t_1,t_2)\mapsto t_2$, respectively. We will also write 
\[
\mathbb{C}[X_*(T)]\cong\mathbb{C}\left[D_{i,1}^{\pm1},D_{i,2}^{\pm1}\right]
\]
where $D_{i,1}$ and~$D_{i,2}$ correspond to the cocharacters of the $i$th factor $T_{\mathrm{GL}_2}\subset T$ given by $t\mapsto\diag(t,1)$ and $t\mapsto\diag(1,t)$, respectively. Thus we see that $K^{(T\times F)_{\mathcal{O}}\rtimes\mathbb{C}^*}(\mathcal{R}_{T,0})$ is generated as a $\mathbb{C}$-algebra by the variables $q^{\pm\frac{1}{2}}$, $z_{k,1}^{\pm1}$, $z_{k,2}^{\pm1}$, $z_a^{\pm1}$, $w_{i,1}^{\pm1}$, $w_{i,2}^{\pm1}$, $D_{i,1}^{\pm1}$, and~$D_{i,2}^{\pm1}$. In fact, if we write $\mathcal{D}_{q,\mathbf{z}}^0$ for the $\mathbb{C}[q^{\pm\frac{1}{2}},z_{k,1}^{\pm1},z_{k,2}^{\pm1},z_a^{\pm1}]$-algebra generated by the variables $w_{i,1}^{\pm1}$, $w_{i,2}^{\pm1}$, $D_{i,1}^{\pm1}$, $D_{i,2}^{\pm1}$, subject to the relations 
\[
[D_{i,r},D_{j,s}]=[w_{i,r},w_{j,s}]=0, \quad D_{i,r}w_{j,s}=q^{2\delta_{ij}\delta_{rs}}w_{j,s}D_{i,r},
\]
then we have an isomorphism $K^{(T\times F)_{\mathcal{O}}\rtimes\mathbb{C}^*}(\mathcal{R}_{T,0})\cong\mathcal{D}_{q,\mathbf{z}}^0$ of algebras over~$\mathbb{C}$.

Now suppose we have a representation $N$ of~$G$, and let us use the same notation for the restriction of this representation to the subgroup~$T\subset G$. Then by restricting the embedding $\mathcal{T}_{T,N}\cong T_{\mathcal{K}}\times_{T_{\mathcal{O}}}N_{\mathcal{O}}\hookrightarrow G_{\mathcal{K}}\times_{G_{\mathcal{O}}}N_{\mathcal{O}}\cong\mathcal{T}_{G,N}$, we obtain a map 
\[
\iota:\mathcal{R}_{T,N}\rightarrow\mathcal{R}_{G,N}.
\]
For any $K^{T\times F\rtimes\mathbb{C}^*}(\mathrm{pt})$-module $M$, let us write $M'$ for the localization of~$M$ at the multiplicative set generated by the expressions $w_{i,r}-q^kw_{i,s}$ for $k\in\mathbb{Z}$ and $r\neq s$ and the expressions $1-q^k$ for $k\in\mathbb{Z}\setminus\{0\}$. By the localization theorem, the pushforward $\iota_*$ provides an isomorphism $\iota_*:K^{(T\times F)_{\mathcal{O}}\rtimes\mathbb{C}^*}(\mathcal{R}_{T,N})'\stackrel{\sim}{\rightarrow}K^{(T\times F)_{\mathcal{O}}\rtimes\mathbb{C}^*}(\mathcal{R}_{G,N})'$. We also have $K^{(G\times F)_{\mathcal{O}}\rtimes\mathbb{C}^*}(\mathcal{R}_{G,N})\cong K^{(T\times F)_{\mathcal{O}}\rtimes\mathbb{C}^*}(\mathcal{R}_{G,N})^W$ where $W$ is the Weyl~group, and so $(\iota_*)^{-1}$ restricts to an embedding 
\[
(\iota_*)^{-1}:K^{(G\times F)_{\mathcal{O}}\rtimes\mathbb{C}^*}(\mathcal{R}_{G,N})\hookrightarrow K^{(T\times F)_{\mathcal{O}}\rtimes\mathbb{C}^*}(\mathcal{R}_{T,N})'.
\]
On the other hand, the zero section of the projection $\pi:\mathcal{T}_{T,N}\rightarrow\Gr_T$ defines a map 
\[
\zeta:\mathcal{R}_{T,0}\cong\Gr_T\rightarrow\mathcal{R}_{T,N}
\]
whose pullback $\zeta^*$ provides a map
\[
\zeta^*:K^{(T\times F)_{\mathcal{O}}\rtimes\mathbb{C}^*}(\mathcal{R}_{T,N})'\rightarrow K^{(T\times F)_{\mathcal{O}}\rtimes\mathbb{C}^*}(\mathcal{R}_{T,0})'.
\]
Thus we have an embedding 
\begin{equation}
\label{eqn:embedding}
\zeta^*(\iota_*)^{-1}:K^{(G\times F)_{\mathcal{O}}\rtimes\mathbb{C}^*}(\mathcal{R}_{G,N})\hookrightarrow\mathcal{D}_{q,\mathbf{z}}
\end{equation}
where $\mathcal{D}_{q,\mathbf{z}}=(\mathcal{D}_{q,\mathbf{z}}^0)'$.

\subsection{Changing the gauge and flavor groups}

Let us continue with the notation of the previous subsection, writing $G=\mathrm{GL}_2^m$ and $F=F_1\times F_2$ where $F_1=T_{\mathrm{GL}_2}^n$ and $F_2=(\mathbb{C}^*)^d$. Let us also write $H=\mathrm{SL}_2^m$ and $L=L_1\times L_2$ where $L_1=T_{\mathrm{SL}_2}^n$ and $L_2=(\mathbb{C}^*)^d$.

\begin{lemma}
\label{lem:KSL2GL2}
Take notation as in the previous paragraph. If we regard $K^{(G\times F)_{\mathcal{O}}\rtimes\mathbb{C}^*}(\mathcal{R}_{G,N})$ as a subalgebra of $\mathcal{D}_{q,\mathbf{z}}$ using the embedding~\eqref{eqn:embedding}, then we have 
\[
K^{(H\times L)_{\mathcal{O}}\rtimes\mathbb{C}^*}(\mathcal{R}_{H,N})\cong K^{(G\times F)_{\mathcal{O}}\rtimes\mathbb{C}^*}(\mathcal{R}_{G,N})^{(\mathbb{C}^*)^m}/(z_{k,1}z_{k,2}-1,w_{i,1}w_{i,2}-1)
\]
where the $i$th factor of $(\mathbb{C}^*)^m$ acts by simultaneously rescaling the generators~$D_{i,1}$ and~$D_{i,2}$.
\end{lemma}

\begin{proof}
Let us write $T_{\mathrm{GL}_2}$ and $T_{\mathrm{SL}_2}$ for the diagonal subgroups of~$\mathrm{GL}_2$ and~$\mathrm{SL}_2$, respectively. Recall that $w_{i,1}$,~$w_{i,2}\in\mathbb{C}[X^*(T_{\mathrm{GL}_2}^m)]$ correspond to the characters of the $i$th factor $T_{\mathrm{GL}_2}\subset T_{\mathrm{GL}_2}^m$ given by $\diag(t_1,t_2)\mapsto t_1$ and $\diag(t_1,t_2)\mapsto t_2$. By restricting characters, we get a surjective algebra homomorphism $\mathbb{C}[X^*(T_{\mathrm{GL}_2}^m)]\twoheadrightarrow\mathbb{C}[X^*(T_{\mathrm{SL}_2}^m)]$. The kernel of this map is the ideal generated by expressions of the form $w_{i,1}w_{i,2}-1$, so we have $\mathbb{C}[X^*(T_{\mathrm{SL}_2}^m)]\cong\mathbb{C}[X^*(T_{\mathrm{GL}_2}^m)]/(w_{i,1}w_{i,2}-1)$. Similarly, $\mathbb{C}[X^*(L)]\cong\mathbb{C}[X^*(F)]/(z_{k,1}z_{k,2}-1)$. Now consider the $K$-theory $K^{(T_{\mathrm{SL}_2}^m\times L)_{\mathcal{O}}\rtimes\mathbb{C}^*}(\mathcal{R}_{T_{\mathrm{GL}_2}^m,0})$ defined as in Definition~\ref{def:KtheoryR} by decomposing the variety of triples into finite-dimensional pieces and taking a limit. As in Section~\ref{sec:TheAbelianCase}, we~compute 
\begin{align*}
K^{(T_{\mathrm{SL}_2}^m\times L)_{\mathcal{O}}\rtimes\mathbb{C}^*}(\mathcal{R}_{T_{\mathrm{GL}_2}^m,0}) &\cong \bigoplus_{\mu\in X_*(T_{\mathrm{GL}_2}^m)}K^{T_{\mathrm{SL}_2}^m\times L\rtimes\mathbb{C}^*}(\mathrm{pt}) \\
&\cong\mathcal{D}_{q,\mathbf{z}}^0/(z_{k,1}z_{k,2}-1,w_{i,1}w_{i,2}-1)
\end{align*}
where the last line is the quotient of $\mathcal{D}_{q,\mathbf{z}}^0$ by the right $\mathcal{D}_{q,\mathbf{z}}^0$-submodule generated by all expressions of the form $z_{k,1}z_{k,2}-1$ and $w_{i,1}w_{i,2}-1$.

In Section~\ref{sec:TheAbelianCase}, we described the embedding 
\[
K^{(G\times F)_{\mathcal{O}}\rtimes\mathbb{C}^*}(\mathcal{R}_{G,N})\hookrightarrow\mathcal{D}_{q,\mathbf{z}}
\]
where the right hand side denotes the localization at the multiplicative set generated by expressions of the form $w_{i,1}-q^kw_{i,2}$ and $1-q^k$. The same construction provides an embedding 
\begin{equation}
\label{eqn:modifiedembedding}
K^{(H\times L)_{\mathcal{O}}\rtimes\mathbb{C}^*}(\mathcal{R}_{G,N})\hookrightarrow\mathcal{D}_{q,\mathbf{z}}/(z_{k,1}z_{k,2}-1,w_{i,1}w_{i,2}-1)
\end{equation}
where the left hand side is defined similarly to Definition~\ref{def:KtheoryR}.

Note that if $X$ is any variety with an action of an algebraic group~$\Gamma$ and $\Pi\leq\Gamma$ is a closed subgroup, then any $\Gamma$-equivariant coherent sheaf on~$X$ can be regarded as an $\Pi$-equivariant coherent sheaf. Since the equivariant $K$-theory of~$X$ is defined as the Grothendieck group of the category of equivariant coherent sheaves on~$X$, it~follows that there is a map $K^\Gamma(X)\rightarrow K^\Pi(X)$. Therefore, by Definition~\ref{def:KtheoryR}, we obtain the maps $K^{(G\times F)_{\mathcal{O}}\rtimes\mathbb{C}^*}(\mathcal{R}_{G,N})\rightarrow K^{(H\times L)_{\mathcal{O}}\rtimes\mathbb{C}^*}(\mathcal{R}_{G,N})$ and $K^{(T_{\mathrm{GL}_2}^m\times F)_{\mathcal{O}}\rtimes\mathbb{C}^*}(\mathcal{R}_{T_{\mathrm{GL}_2}^m,0})\rightarrow K^{(T_{\mathrm{SL}_2}^m\times L)_{\mathcal{O}}\rtimes\mathbb{C}^*}(\mathcal{R}_{T_{\mathrm{GL}_2}^m,0})$. The first of these maps, together with the localization of the second, provide the vertical arrows of the commutative diagram 
\[
\xymatrix{
K^{(G\times F)_{\mathcal{O}}\rtimes\mathbb{C}^*}(\mathcal{R}_{G,N}) \ar[d] \ar@{^{(}->}[r] & \mathcal{D}_{q,\mathbf{z}} \ar[d] \\
K^{(H\times L)_{\mathcal{O}}\rtimes\mathbb{C}^*}(\mathcal{R}_{G,N}) \ar@{^{(}->}[r] & \mathcal{D}_{q,\mathbf{z}}/(z_{k,1}z_{k,2}-1,w_{i,1}w_{i,2}-1).
}
\]
Concretely, the vertical arrow on the right hand side is the quotient by expressions of the form $w_{i,1}w_{i,2}-1$ and $z_{k,1}z_{k,2}-1$, and the vertical arrow on the left hand side is the natural map $K^{(G\times F)_{\mathcal{O}}\rtimes\mathbb{C}^*}(\mathcal{R}_{G,N})\rightarrow K^{(G\times F)_{\mathcal{O}}\rtimes\mathbb{C}^*}(\mathcal{R}_{G,N})\otimes_{K^{(G\times F)_{\mathcal{O}}\rtimes\mathbb{C}^*}(\mathrm{pt})}K^{(H\times L)_{\mathcal{O}}\rtimes\mathbb{C}^*}(\mathrm{pt})$, which is easily seen to be surjective. It follows that there is an isomorphism 
\begin{equation}
\label{eqn:KSL2GL2}
K^{(H\times L)_{\mathcal{O}}\rtimes\mathbb{C}^*}(\mathcal{R}_{G,N})\cong K^{(G\times F)_{\mathcal{O}}\rtimes\mathbb{C}^*}(\mathcal{R}_{G,N})/(z_{k,1}z_{k,2}-1,w_{i,1}w_{i,2}-1)
\end{equation}
where the right hand side is the quotient of the image of $K^{(G\times F)_{\mathcal{O}}\rtimes\mathbb{C}^*}(\mathcal{R}_{G,N})$ by the right submodule generated by $z_{k,1}z_{k,2}-1$ and $w_{i,1}w_{i,2}-1$.

Now the connected components of $\Gr_G$ are parametrized by the fundamental group $\pi_1(G)$. The space $\mathcal{R}_{G,N}$ is homotopy equivalent to~$\Gr_G$, so we have a decomposition $\mathcal{R}_{G,N}=\coprod_{\gamma\in\pi_1(G)}\mathcal{R}_\gamma$ into connected components $\mathcal{R}_\gamma$. There is an associated direct sum decomposition 
\[
K^{(H\times L)_{\mathcal{O}}\rtimes\mathbb{C}^*}(\mathcal{R}_{G,N})=\bigoplus_{\gamma\in\pi_1(G)}K^{(H\times L)_{\mathcal{O}}\rtimes\mathbb{C}^*}(\mathcal{R}_\gamma)
\]
so that the $K$-theory of $\mathcal{R}_{G,N}$ is graded by~$\pi_1(G)\cong\mathbb{Z}^m$. Thus there is a natural action of~$(\mathbb{C}^*)^m$ on this $K$-theory. The total degrees in the variables~$D_{i,1}$,~$D_{i,2}$ for $i=1,\dots,m$ define a $\mathbb{Z}^m$-grading of~$\mathcal{D}_{q,\mathbf{z}}/(z_{k,1}z_{k,2}-1,w_{i,1}w_{i,2}-1)$ and an associated $(\mathbb{C}^*)^m$-action. The embedding~\eqref{eqn:modifiedembedding} preserves the gradings and therefore intertwines the $(\mathbb{C}^*)^m$-actions. The affine Grassmannian $\Gr_H$ is identified with the component of $\Gr_G$ corresponding to $1\in\pi_1(G)$, so there is an identification $\mathcal{R}_{H,N}\cong\mathcal{R}_1$. Hence, by equation~\eqref{eqn:KSL2GL2}, we have 
\begin{align*}
K^{(H\times L)_{\mathcal{O}}\rtimes\mathbb{C}^*}(\mathcal{R}_{H,N}) &\cong K^{(H\times L)_{\mathcal{O}}\rtimes\mathbb{C}^*}(\mathcal{R}_{G,N})^{(\mathbb{C}^*)^m} \\
&\cong K^{(G\times F)_{\mathcal{O}}\rtimes\mathbb{C}^*}(\mathcal{R}_{G,N})^{(\mathbb{C}^*)^m}/(z_{k,1}z_{k,2}-1,w_{i,1}w_{i,2}-1).
\end{align*}
The last line is the quotient of $K^{(G\times F)_{\mathcal{O}}\rtimes\mathbb{C}^*}(\mathcal{R}_{G,N})^{(\mathbb{C}^*)^m}$ by the right submodule generated by $z_{k,1}z_{k,2}-1$ and $w_{i,1}w_{i,2}-1$, which is in fact a two-sided ideal in this ring.
\end{proof}

\subsection{Monopole operators}
\label{sec:MonopoleOperators}

Let $Q=(Q_0,Q_1,s,t)$ be a finite quiver where $Q_0$ is the set of vertices, $Q_1$ is the set of arrows, and $s:Q_1\rightarrow Q_0$ and $t:Q_1\rightarrow Q_0$ are the source and target maps, respectively. Let us fix a decomposition $Q_0=\mathcal{G}\sqcup\mathcal{F}$ of the set of vertices into a set $\mathcal{G}$ of \emph{gauge~nodes} and a set $\mathcal{F}$ of \emph{framing~nodes} such that no arrow of~$Q$ ends at a framing node. We also fix a pair of finite-dimensional complex vector spaces 
\[
V=\bigoplus_{i\in\mathcal{G}}V_i, \quad U=\bigoplus_{i\in\mathcal{F}}U_i
\]
graded by these sets~$\mathcal{G}$ and~$\mathcal{F}$. From these we construct the vector space 
\begin{equation}
\label{eqn:quiverrep}
N=\bigoplus_{s(a)\in\mathcal{G}}\Hom_{\mathbb{C}}(V_{s(a)},V_{t(a)})\oplus\bigoplus_{s(a)\in\mathcal{F}}\Hom_{\mathbb{C}}(U_{s(a)},V_{t(a)})
\end{equation}
where the first sum is over all arrows $a\in Q_1$ such that $s(a)$ is a gauge node and the second sum is over all $a\in Q_1$ such that $s(a)$ is a framing node.

Next we define the gauge group 
\[
G=\prod_{i\in\mathcal{G}}\mathrm{GL}(V_i),
\]
which acts naturally on the vector space~$N$. Choosing a basis for each~$U_i$, we obtain a diagonal subgroup $F_1\subset\prod_i\mathrm{GL}(U_i)$, which acts naturally on~$N$. We also consider the torus $F_2=(\mathbb{C}^*)^{Q_1}$ determined by~$Q$. There is an action of $F_2$ on~$N$ where the $\mathbb{C}^*$-factor corresponding to an arrow $a\in Q_1$ acts by rescaling the corresponding hom-set in the direct sum decomposition~\eqref{eqn:quiverrep}. Then the full flavor symmetry group is the product 
\[
F=F_1\times F_2.
\]
The actions of~$G$ and~$F$ on the vector space~$N$ commute, and hence we obtain an action of the extended gauge group $\widetilde{G}=G\times F$.

We can now define generators for the quantized Coulomb branch of~$(\widetilde{G},N)$. We use the notation $c_i\coloneqq\dim U_i$ and $d_i\coloneqq\dim V_i$. To each framing node $k\in\mathcal{F}$, we associate variables~$z_{k,l}$ for $l=1,\dots,c_k$, and to each $a\in Q_1$, we associate a variable~$z_a$. We then write $\mathcal{D}_{q,\mathbf{z}}^0$ for the $\mathbb{C}[q^{\pm\frac{1}{2}},z_{k,l}^{\pm1},z_a^{\pm1}]$-algebra generated by variables~$D_{i,r}$,~$w_{i,r}$ and their inverses, for $i\in\mathcal{G}$ and $r=1,\dots,d_i$, subject to the commutation relations 
\[
[D_{i,r},D_{j,s}]=[w_{i,r},w_{j,s}]=0, \quad D_{i,r}w_{j,s}=q^{2\delta_{ij}\delta_{rs}}w_{j,s}D_{i,r}.
\]
We write $\mathcal{D}_{q,\mathbf{z}}$ for the localization of this algebra $\mathcal{D}_{q,\mathbf{z}}^0$ at the multiplicative set generated by the expressions $w_{i,r}-q^\mu w_{i,s}$ for $\mu\in\mathbb{Z}$ and $r\neq s$ and the expressions $1-q^\mu$ for $\mu\in\mathbb{Z}\setminus\{0\}$. Note that our $q$ corresponds to the variable denoted~$\boldsymbol{v}$ in~\cite{FT19a,FT19b}.

\begin{definition}[\cite{BFN19a,FT19a,SS19}]
\label{def:monopole}
Fix $i\in\mathcal{G}$ and let $\mathcal{S}_i=\{a\in Q_1:s(a)\in\mathcal{F},t(a)=i\}$. For each integer $n=1,\dots,d_i$ and each symmetric Laurent polynomial~$f$ in $n$~variables, we define elements $E_{i,n}[f]$,~$F_{i,n}[f]\in\mathcal{D}_{q,\mathbf{z}}$, called \emph{dressed minuscule monopole operators}, by the formulas 
\[
E_{i,n}[f] = \sum_{\substack{\mathcal{I}\subset\{1,\dots,d_i\} \\ |\mathcal{I}|=n}}f(w_{i,\mathcal{I}})\mathcal{P}_{i,\mathcal{I}}\prod_{r\in\mathcal{I}}D_{i,r}
\]
and 
\[
F_{i,n}[f] = \sum_{\substack{\mathcal{I}\subset\{1,\dots,d_i\} \\ |\mathcal{I}|=n}}f(q^{-2}w_{i,\mathcal{I}})\prod_{\substack{a\in\mathcal{S}_i \\  r\in \mathcal{I}}}\prod_l(1-qz_{s(a),l}z_aw_{i,r}^{-1})\mathcal{Q}_{i,\mathcal{I}}\prod_{r\in\mathcal{I}}D_{i,r}^{-1}
\]
where 
\[
\mathcal{P}_{i,\mathcal{I}} =\frac{\prod_{\substack{a:i\rightarrow j \\ r\in\mathcal{I}}}\prod_{s:s\not\in\mathcal{I}\text{ if }i=j}(1-qz_aw_{i,r}w_{j,s}^{-1})}{\prod_{\substack{r\in\mathcal{I} \\ s\not\in\mathcal{I}}}(1-w_{i,s}w_{i,r}^{-1})},
\qquad
\mathcal{Q}_{i,\mathcal{I}} =\frac{\prod_{\substack{a:j\rightarrow i \\ r\in\mathcal{I}}}\prod_{s:s\not\in\mathcal{I}\text{ if }i=j}(1-qz_aw_{j,s}w_{i,r}^{-1})}{\prod_{\substack{r\in\mathcal{I} \\ s\not\in\mathcal{I}}}(1-w_{i,r}w_{i,s}^{-1})},
\]
and we write $f(w_{i,\mathcal{I}})$ (respectively, $f(q^{-2}w_{i,\mathcal{I}})$) for the element of $\mathcal{D}_{q,\mathbf{z}}$ obtained by substituting the elements $(w_{i,r})_{r\in\mathcal{I}}$ (respectively, $(q^{-2}w_{i,r})_{r\in\mathcal{I}}$) into the argument of~$f$. The function~$f$ is called the \emph{dressing} for $E_{i,n}[f]$ and $F_{i,n}[f]$.
\end{definition}

\subsection{Generating the Coulomb branch}
\label{sec:GeneratingTheCoulombBranch}

We now specialize the discussion of the last subsection to the case where $U_i=\mathbb{C}^2$ for all $i\in\mathcal{F}$ and $V_i=\mathbb{C}^2$ for all $i\in\mathcal{G}$. In this case, the algebra $\mathcal{D}_{q,\mathbf{z}}$ specializes to the one considered in Section~\ref{sec:TheAbelianCase}. We write~$\mathcal{A}_{q,\mathbf{z}}$ for the $\mathbb{C}[q^{\pm\frac{1}{2}},z_{k,l}^{\pm1},z_a^{\pm1}]$-subalgebra generated by all dressed minuscule monopole operators together with all symmetric Laurent polynomials in the variables~$w_{i,r}$ for each~$i$.

We are interested in particular in the case where $Q$ is one of the quivers illustrated in Figure~\ref{fig:twoquivers}. When drawing pictures of quivers, we represent gauge nodes by circles and framing nodes by squares. The number inside a circle (respectively, square) representing a node~$i$ is the dimension of the vector space~$V_i$ (respectively,~$U_i$) assigned to this node.

\begin{figure}[ht]
\begin{subfigure}{0.3\textwidth}
\begin{center}
\[
\xymatrix{
*+[F]{2} \save +<3mm,0mm>\ar[rr]+<-3mm,0mm>\restore & & *+[o][F-]{2} & & *+[F]{2} \save +<-3mm,0mm>\ar[ll]+<3mm,0mm>\restore
}
\]
\end{center}
\caption{\label{subfig:typeA}}
\end{subfigure}
\begin{subfigure}{0.3\textwidth}
\begin{center}
\[
\xymatrix{
*+[o][F-]{2} \ar@(ur,ul)
}
\]
\end{center}
\caption{\label{subfig:Jordan}}
\end{subfigure}
\caption{Quivers in Proposition~\ref{prop:generate}.\label{fig:twoquivers}}
\end{figure}
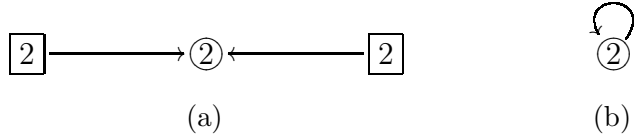

We have the following description of the quantized Coulomb branches defined by these quivers (see~\cite{W19} for a more general homology version of this statement).

\begin{proposition}
\label{prop:generate}
Let $(\widetilde{G},N)$ be the group and representation defined by either of the quivers illustrated in Figure~\ref{fig:twoquivers}. Then the embedding~\eqref{eqn:embedding} provides an isomorphism of algebras 
\[
K^{(G\times F)_{\mathcal{O}}\rtimes\mathbb{C}^*}(\mathcal{R}_{G,N})\cong\mathcal{A}_{q,\mathbf{z}}.
\]
\end{proposition}

\begin{proof}
For the quiver illustrated in Figure~\ref{subfig:typeA}, the statement was proved in Theorem~4.32 of~\cite{FT19b}. For the quiver illustrated in Figure~\ref{subfig:Jordan}, the statement can be proved in the same way as this theorem. In fact, the proof simplifies considerably in this example because the ``generalized roots'' considered in the proof are the ordinary roots and ``chambers'' are the usual Weyl~chambers for~$\mathrm{GL}_2$. Consequently, the argument comes down to the fact that any dominant coweight is a sum of fundamental coweights.
\end{proof}

\section{Proofs of the main results}
\label{sec:ProofsOfTheMainResults}

In this section, we formulate and prove our main results relating skein algebras and quantized Coulomb branches.

\subsection{Representations from surfaces}
\label{sec:RepresentationsFromSurfaces}

Recall that a \emph{pair of pants} is a compact surface obtained from~$S^2$ by removing three open disks with embedded disjoint closures. If $S$ is a compact surface with negative Euler characteristic, then a \emph{pants decomposition} of~$S$ is a collection of disjoint simple closed curves on~$S$ such that any component of the complement of these curves in~$S\setminus\partial S$ is the interior of a pair of pants. Let~$\mathcal{P}$ be a pants decomposition of~$S$, and let $\mathcal{B}$ be the set of all loops which are boundary components of~$S$. Then there is a collection~$\{P_i\}$ of pairs of pants such that the surface $S$ is obtained by identifying boundary components of $\coprod_iP_i$ and each boundary component of $P_i$ corresponds to a curve in~$\mathcal{P}\cup\mathcal{B}$. Note that it is possible for two boundary components of a single~$P_i$ to be identified in~$S$.

We will now construct a representation associated to the pants decomposition~$\mathcal{P}$. This representation is closely related to the one in~\cite{G12} and has also appeared in Section~5 of~\cite{BFN19b}. We start with a two-dimensional vector space $M_\gamma=\mathbb{C}^2$ associated to each curve $\gamma\in\mathcal{P}\cup\mathcal{B}$. If $P_i$ is a pair of pants whose boundary components correspond to the curves $\gamma_{i1}$,~$\gamma_{i2}$,~$\gamma_{i3}\in\mathcal{P}\cup\mathcal{B}$, then we form the tensor product 
\[
M_{P_i}=M_{\gamma_{i1}}\otimes M_{\gamma_{i2}}\otimes M_{\gamma_{i3}}
\]
of the associated vector spaces. We let 
\[
M=\bigoplus_iM_{P_i}
\]
be the direct sum over all pairs of pants.

For each $\gamma\in\mathcal{P}$ let $G_\gamma=\mathrm{SL}_2(\mathbb{C})$, and for each $\gamma\in\mathcal{B}$ let $F_\gamma\subset\mathrm{SL}_2(\mathbb{C})$ be the diagonal torus. There is an action of~$G_\gamma$ on~$M_{P_i}$ where an element of~$G_\gamma$ acts by matrix multiplication on each factor equal to~$M_\gamma$ in the tensor product and acts trivially on all other tensor factors. By taking the direct sum over all~$i$, we get an action of~$G_\gamma$ on~$M$. In the same way, if $\gamma\in\mathcal{B}$, then we get an action of~$F_\gamma$ on~$M_{P_i}$, and taking the direct sum over all~$i$ gives an action of~$F_\gamma$ on~$M$. Thus we see that the groups 
\[
G=\prod_{\gamma\in\mathcal{P}}G_\gamma, \qquad F=\prod_{\gamma\in\mathcal{B}}F_\gamma
\]
act naturally on the vector space~$M$. One easily sees that these actions commute, and so we obtain an action of the product $\widetilde{G}=G\times F$ on~$M$.

We will be interested in examples where this representation $M$ is of \emph{cotangent~type}, that~is, $M=N\oplus N^*$ for some representation~$N$ of~$\widetilde{G}$. To understand how such representations can arise, consider a pair of pants $P_i$ whose boundary components correspond to curves $\gamma_{i1}$,~$\gamma_{i2}$,~$\gamma_{i3}\in\mathcal{P}\cup\mathcal{B}$. If one of these curves, say $\gamma_{i1}$, is a boundary component of~$S$, then we can write 
\[
M_{P_i}\cong(M_{\gamma_{i2}}\otimes M_{\gamma_{i3}})\oplus(M_{\gamma_{i2}}\otimes M_{\gamma_{i3}})
\]
where an element $\diag(t,t^{-1})\in F_{\gamma_{i1}}$ acts on~$M_{P_i}$ by multiplying the first direct summand by $t$ and multiplying the second direct summand by~$t^{-1}$. Let us write $N_{P_i}=M_{\gamma_{i2}}\otimes M_{\gamma_{i3}}$. This is a representation of $G_{\gamma_{i2}}\times G_{\gamma_{i3}}$ or $F_{\gamma_{i2}}\times G_{\gamma_{i3}}$ or $G_{\gamma_{i2}}\times F_{\gamma_{i3}}$ or $F_{\gamma_{i2}}\times F_{\gamma_{i3}}$, depending on whether $\gamma_{i2}$ and~$\gamma_{i3}$ belong to $\mathcal{P}$ or~$\mathcal{B}$. Since the matrix representations of $\mathrm{SL}_2(\mathbb{C})$ and its diagonal subgroup are self-dual, we see that in all cases we have $N_{P_i}\cong N_{P_i}^*$ as representations.

It follows that if each pair of pants~$P_i$ shares at least one boundary component with~$S$, then $M=N\oplus N^*$ where $N=\bigoplus_iN_{P_i}$, and hence the representation~$M$ is of cotangent type. In~particular, when the genus of~$S$ is zero or one, we can get a representation~$M$ of cotangent type by choosing the pants decomposition as in Figure~\ref{fig:pantsdecompositions}.

\begin{figure}[ht]
\begin{subfigure}{0.5\textwidth}
\begin{center}
\begin{tikzpicture}[scale=1.5]
\clip(-2.75,-1.25) rectangle (2.75,1.25);
\draw[black, thin] plot [smooth, tension=1] coordinates { (2.5,-0.125) (2.25,0) (2.5,0.125)};
\draw[black, thin] plot [smooth, tension=1] coordinates { (-2.5,-0.125) (-2.25,0) (-2.5,0.125)};
\draw[black, thin] plot [smooth, tension=1] coordinates { (1.5,0.25) (1.8,0.3) (2.2,0.55) (2.5,0.625)};
\draw[black, thin] plot [smooth, tension=1] coordinates { (1.5,-0.25) (1.8,-0.3) (2.2,-0.55) (2.5,-0.625)};
\draw[black, thin] plot [smooth, tension=1] coordinates { (-1.5,0.25) (-1.8,0.3) (-2.2,0.55) (-2.5,0.625)};
\draw[black, thin] plot [smooth, tension=1] coordinates { (-1.5,-0.25) (-1.8,-0.3) (-2.2,-0.55) (-2.5,-0.625)};
\draw[black, thin] plot [smooth, tension=1] coordinates { (0.5,0.25) (0.7,0.3) (0.75,0.5)};
\draw[black, thin] plot [smooth, tension=1] coordinates { (1.5,0.25) (1.3,0.3) (1.25,0.5)};
\draw[black, thin] plot [smooth, tension=1] coordinates { (-0.5,0.25) (-0.7,0.3) (-0.75,0.5)};
\draw[black, thin] plot [smooth, tension=1] coordinates { (-1.5,0.25) (-1.3,0.3) (-1.25,0.5)};
\draw (-0.4,-0.25) -- (-1.5,-0.25);
\draw (0.4,-0.25) -- (1.5,-0.25);
\draw (-0.4,0.25) -- (-0.5,0.25);
\draw (0.4,0.25) -- (0.5,0.25);
\draw[black, thin] plot [smooth, tension=1] coordinates { (-0.75,0.5) (-0.875,0.45) (-1.125,0.45) (-1.25,0.5)};
\draw[black, thin] plot [smooth, tension=1] coordinates { (-0.75,0.5) (-0.875,0.55) (-1.125,0.55) (-1.25,0.5)};
\draw[black, thin] plot [smooth, tension=1] coordinates { (0.75,0.5) (0.875,0.45) (1.125,0.45) (1.25,0.5)};
\draw[black, thin] plot [smooth, tension=1] coordinates { (0.75,0.5) (0.875,0.55) (1.125,0.55) (1.25,0.5)};
\draw[black, thin] plot [smooth, tension=1] coordinates { (-2.5,0.125) (-2.45,0.2) (-2.45,0.5) (-2.5,0.625)};
\draw[black, thin] plot [smooth, tension=1] coordinates { (-2.5,0.125) (-2.55,0.2) (-2.55,0.5) (-2.5,0.625)};
\draw[black, thin] plot [smooth, tension=1] coordinates { (-2.5,-0.125) (-2.45,-0.2) (-2.45,-0.5) (-2.5,-0.625)};
\draw[black, thin] plot [smooth, tension=1] coordinates { (-2.5,-0.125) (-2.55,-0.2) (-2.55,-0.5) (-2.5,-0.625)};
\draw[black, thin, dotted] plot [smooth, tension=1] coordinates { (-1.5,-0.25) (-1.55,-0.125) (-1.55,0.125) (-1.5,0.25)};
\draw[black, thin] plot [smooth, tension=1] coordinates { (-1.5,-0.25) (-1.45,-0.125) (-1.45,0.125) (-1.5,0.25)};
\draw[black, thin, dotted] plot [smooth, tension=1] coordinates { (-0.5,-0.25) (-0.55,-0.125) (-0.55,0.125) (-0.5,0.25)};
\draw[black, thin] plot [smooth, tension=1] coordinates { (-0.5,-0.25) (-0.45,-0.125) (-0.45,0.125) (-0.5,0.25)};
\draw[black, thin, dotted] plot [smooth, tension=1] coordinates { (0.5,-0.25) (0.45,-0.125) (0.45,0.125) (0.5,0.25)};
\draw[black, thin] plot [smooth, tension=1] coordinates { (0.5,-0.25) (0.55,-0.125) (0.55,0.125) (0.5,0.25)};
\draw[black, thin, dotted] plot [smooth, tension=1] coordinates { (1.5,-0.25) (1.45,-0.125) (1.45,0.125) (1.5,0.25)};
\draw[black, thin] plot [smooth, tension=1] coordinates { (1.5,-0.25) (1.55,-0.125) (1.55,0.125) (1.5,0.25)};
\draw[black, thin, dotted] plot [smooth, tension=1] coordinates { (2.5,0.125) (2.45,0.2) (2.45,0.5) (2.5,0.625)};
\draw[black, thin] plot [smooth, tension=1] coordinates { (2.5,0.125) (2.55,0.2) (2.55,0.5) (2.5,0.625)};
\draw[black, thin, dotted] plot [smooth, tension=1] coordinates { (2.5,-0.125) (2.45,-0.2) (2.45,-0.5) (2.5,-0.625)};
\draw[black, thin] plot [smooth, tension=1] coordinates { (2.5,-0.125) (2.55,-0.2) (2.55,-0.5) (2.5,-0.625)};
\node at (0,0) {$\dots$};
\end{tikzpicture}
\end{center}
\caption{\label{subfig:genus0}}
\end{subfigure}
\begin{subfigure}{0.23\textwidth}
\begin{center}
\begin{tikzpicture}[scale=1.5]
\clip(-1.25,-1.25) rectangle (1.25,1.25);
\begin{scope}
\clip[rotate=60](0.5,0) rectangle (1,0.05);
\draw[black, thin, dotted, rotate=60] (0.75,0) ellipse (0.25 and 0.05);
\end{scope}
\begin{scope}
\clip[rotate=60](0.5,-0.05) rectangle (1,0);
\draw[black, thin, rotate=60] (0.75,0) ellipse (0.25 and 0.05);
\end{scope}
\draw[black, thin, rotate=90] (1.125,0) ellipse (0.05 and 0.25);
\begin{scope}
\clip[rotate=120](0.5,0) rectangle (1,0.05);
\draw[black, thin, dotted, rotate=120] (0.75,0) ellipse (0.25 and 0.05);
\end{scope}
\begin{scope}
\clip[rotate=120](0.5,-0.05) rectangle (1,0);
\draw[black, thin, rotate=120] (0.75,0) ellipse (0.25 and 0.05);
\end{scope}
\draw[black, thin, rotate=150] (1.125,0) ellipse (0.05 and 0.25);
\begin{scope}
\clip[rotate=180](0.5,0) rectangle (1,0.05);
\draw[black, thin, rotate=180] (0.75,0) ellipse (0.25 and 0.05);
\end{scope}
\begin{scope}
\clip[rotate=180](0.5,-0.05) rectangle (1,0);
\draw[black, thin, dotted, rotate=180] (0.75,0) ellipse (0.25 and 0.05);
\end{scope}
\draw[black, thin, rotate=-150] (1.125,0) ellipse (0.05 and 0.25);
\begin{scope}
\clip[rotate=-120](0.5,0) rectangle (1,0.05);
\draw[black, thin, rotate=-120] (0.75,0) ellipse (0.25 and 0.05);
\end{scope}
\begin{scope}
\clip[rotate=-120](0.5,-0.05) rectangle (1,0);
\draw[black, thin, dotted, rotate=-120] (0.75,0) ellipse (0.25 and 0.05);
\end{scope}
\begin{scope}
\clip[rotate=-90](1.125,-0.25) rectangle (1.2,0.25);
\draw[black, thin, rotate=-90] (1.125,0) ellipse (0.05 and 0.25);
\end{scope}
\begin{scope}
\clip[rotate=-90](1.05,-0.25) rectangle (1.125,0.25);
\draw[black, thin, dotted, rotate=-90] (1.125,0) ellipse (0.05 and 0.25);
\end{scope}
\begin{scope}
\clip[rotate=-60](0.5,0) rectangle (1,0.05);
\draw[black, thin, rotate=-60] (0.75,0) ellipse (0.25 and 0.05);
\end{scope}
\begin{scope}
\clip[rotate=-60](0.5,-0.05) rectangle (1,0);
\draw[black, thin, dotted, rotate=-60] (0.75,0) ellipse (0.25 and 0.05);
\end{scope}
\draw [thin,domain=30:330] plot ({0.5*cos(\x)}, {0.5*sin(\x)});
\draw[black, thin] plot [smooth, tension=1] coordinates { (-1,0) (-0.97,0.26) (-1.10,0.35)};
\draw[black, thin] plot [smooth, tension=1] coordinates { (-1,0) (-0.97,-0.26) (-1.10,-0.35)};
\draw[black, thin] plot [smooth, tension=1] coordinates { (-0.5,0.87) (-0.71,0.71) (-0.85,0.78)};
\draw[black, thin] plot [smooth, tension=1] coordinates { (-0.5,-0.87) (-0.71,-0.71) (-0.85,-0.78)};
\draw[black, thin] plot [smooth, tension=1] coordinates { (0.5,0.87) (0.71,0.71)};
\draw[black, thin] plot [smooth, tension=1] coordinates { (0.5,-0.87) (0.71,-0.71)};
\draw[black, thin] plot [smooth, tension=1] coordinates { (-0.5,0.87) (-0.26,0.97) (-0.25,1.125)};
\draw[black, thin] plot [smooth, tension=1] coordinates { (0.5,0.87) (0.26,0.97) (0.25,1.125)};
\draw[black, thin] plot [smooth, tension=1] coordinates { (-0.5,-0.87) (-0.26,-0.97) (-0.25,-1.125)};
\draw[black, thin] plot [smooth, tension=1] coordinates { (0.5,-0.87) (0.26,-0.97) (0.25,-1.125)};
\node at (0.75,0) {$\vdots$};
\end{tikzpicture}
\end{center}
\caption{\label{subfig:genus1}}
\end{subfigure}
\caption{Pants decompositions in genus zero and one.\label{fig:pantsdecompositions}}
\end{figure}
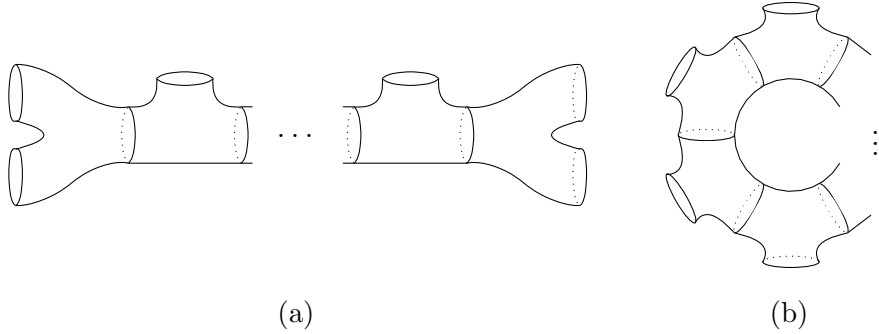

\subsection{The three-holed sphere}

The simplest surface we consider is the three-holed sphere~$S_{0,3}$. This surface has a unique (empty) pants decomposition, which by the above construction yields a group $\widetilde{G}$ and representation~$N$.

\begin{theorem}
\label{thm:S03main}
Let $(\widetilde{G},N)$ be the group and representation associated to the three-holed sphere~$S_{0,3}$. Then there is a $\mathbb{C}$-algebra isomorphism 
\[
\Sk_{A,\boldsymbol{\lambda}}(S_{0,3})\cong K^{\widetilde{G}_{\mathcal{O}}\rtimes\mathbb{C}^*}(\mathcal{R}_{G,N}).
\]
\end{theorem}

\begin{proof}
By the construction of Section~\ref{sec:RepresentationsFromSurfaces}, we see that the gauge group is the trivial group $G\cong\{1\}$, and the representation is $N=\mathbb{C}^2\otimes\mathbb{C}^2\cong\mathbb{C}^4$. Therefore the variety of triples is $\mathcal{R}_{G,N}\cong N_\mathcal{O}\cong\mathcal{O}^4$. The flavor symmetry group is $F\cong(\mathbb{C}^*)^3$, so by Definition~\ref{def:KtheoryR}, we have 
\[
K^{\widetilde{G}_{\mathcal{O}}\rtimes\mathbb{C}^*}(\mathcal{R}_{G,N})\cong K^{(\mathbb{C}^*)^3\rtimes\mathbb{C}^*}(\mathbb{C}^{4d})
\]
for $d\gg0$. We can view $\mathbb{C}^{4d}$ as a $(\mathbb{C}^*)^3\rtimes\mathbb{C}^*$-equivariant vector bundle over a point~$\mathrm{pt}$. By the Thom~isomorphism theorem (\cite{CG97}, Theorem~5.4.17), we get $\mathbb{C}$-algebra isomorphisms 
\[
K^{\widetilde{G}_{\mathcal{O}}\rtimes\mathbb{C}^*}(\mathcal{R}_{G,N})\cong K^{(\mathbb{C}^*)^3\rtimes\mathbb{C}^*}(\mathrm{pt})\cong\mathbb{C}\left[q^{\pm\frac{1}{2}},z_1^{\pm1},z_2^{\pm1},z_3^{\pm1}\right].
\]
On the other hand, let~$\delta_1$,~$\delta_2$,~$\delta_3$ be curves on~$S_{0,3}$ homotopic to the three boundary components. They define elements of the relative skein algebra~$\Sk_{A,\boldsymbol{\lambda}}(S_{0,3})$. Since any simple closed curve in~$S_{0,3}$ is homotopic to one of the~$\delta_i$, we see that the~$\delta_i$ generate $\Sk_{A,\boldsymbol{\lambda}}(S_{0,3})$ as a $\mathbb{C}[A^{\pm1},\lambda_1^{\pm1},\lambda_2^{\pm1},\lambda_3^{\pm1}]$-algebra. In the relative skein algebra, we have the relations $\delta_i=-(\lambda_i+\lambda_i^{-1})$ for~$i=1,2,3$, so we see that there is an isomorphism of $\mathbb{C}$-algebras 
\[
\Sk_{A,\boldsymbol{\lambda}}(S_{0,3})\cong\mathbb{C}[A^{\pm1},\lambda_1^{\pm1},\lambda_2^{\pm1},\lambda_3^{\pm1}].
\]
The theorem follows.
\end{proof}

\subsection{The four-holed sphere}
\label{sec:TheFourHoledSphere}

To understand the Coulomb branch of the representation associated to a four-holed sphere, we consider the following quiver: 
\begin{equation}
\label{eqn:S04quiver}
\xymatrix{
*+[F]{2} \save +<3mm,0mm>\ar[rr]+<-3mm,0mm>\restore & & *+[o][F-]{2} & & *+[F]{2} \save +<-3mm,0mm>\ar[ll]+<3mm,0mm>\restore
}
\end{equation}
By applying the construction of Section~\ref{sec:MonopoleOperators} to this quiver, we get $\mathbb{C}[q^{\pm\frac{1}{2}},z_{k,l}^{\pm1},z_a^{\pm1},z_b^{\pm1}]$-algebras~$\mathcal{D}_{q,\mathbf{z}}^0$ and~$\mathcal{D}_{q,\mathbf{z}}$ where $k\in\{1,2\}$ indexes the framing nodes of the quiver, $a$ and~$b$ are the two arrows of the quiver, and $l\in\{1,2\}$. These algebras are generated by variables $w_r$, $D_r$, and their inverses where $r\in\{1,2\}$. Here we write only the $r$ index since the quiver has only one gauge node. There are $\mathbb{C}^*$-actions on~$\mathcal{D}_{q,\mathbf{z}}^0$ and~$\mathcal{D}_{q,\mathbf{z}}$ that simultaneously rescale the~$D_r$, and we will write $(\mathcal{D}_{q,\mathbf{z}}^0)^{\mathbb{C}^*}$ and~$\mathcal{D}_{q,\mathbf{z}}^{\mathbb{C}^*}$ for the corresponding invariant subalgebras.

\begin{lemma}
\label{lem:relateambientringsS04}
Define the operator $\varpi\in\End\mathbb{C}_{q,\mathbf{t}}(X)$ by $(\varpi f)(X)=f(qX)$ for $f\in\mathbb{C}_{q,\mathbf{t}}(X)$. Then the assignments 
\begin{gather*}
q^{\frac{1}{2}}\mapsto q^{\frac{1}{2}}, \quad z_{1,1}\mapsto t_4, \quad z_{1,2}\mapsto t_4^{-1}, \quad z_{2,1}\mapsto t_1, \quad z_{2,2}\mapsto t_1^{-1}, \quad z_a\mapsto t_3, \quad z_b\mapsto t_2, \\
w_1\mapsto X, \quad w_2\mapsto X^{-1}, \quad D_1D_2^{-1}\mapsto q^{-2}X^{-2}\varpi^2
\end{gather*}
determine a well defined injective $\mathbb{C}$-algebra homomorphism 
\[
\mathcal{D}_{q,\mathbf{z}}^{\mathbb{C}^*}/(z_{k,1}z_{k,2}-1,w_1w_2-1)\hookrightarrow\End\mathbb{C}_{q,\mathbf{t}}(X).
\]
\end{lemma}

\begin{proof}
One can check that these assignments map the commutation relations in~$(\mathcal{D}_{q,\mathbf{z}}^0)^{\mathbb{C}^*}$ to relations in $\End\mathbb{C}_{q,\mathbf{t}}(X)$ and therefore provide a well defined $\mathbb{C}$-algebra homomorphism $(\mathcal{D}_{q,\mathbf{z}}^0)^{\mathbb{C}^*}\rightarrow\End\mathbb{C}_{q,\mathbf{t}}(X)$. Furthermore, the elements $z_{k,1}z_{k,2}-1$ and $w_1w_2-1$ map to zero, so this induces a well defined $\mathbb{C}$-algebra homomorphism 
\begin{equation}
\label{eqn:relateambientringsS04}
(\mathcal{D}_{q,\mathbf{z}}^0)^{\mathbb{C}^*}/(z_{k,1}z_{k,2}-1,w_1w_2-1)\rightarrow\End\mathbb{C}_{q,\mathbf{t}}(X).
\end{equation}
Suppose $\bar{P}$ lies in the kernel of this map. This $\bar{P}$ is the class of some $P\in(\mathcal{D}_{q,\mathbf{z}}^0)^{\mathbb{C}^*}$, which we can take to be a Laurent polynomial in $q^{\frac{1}{2}}$, $z_{1,1}$, $z_{2,1}$, $z_a$, $z_b$, $w_1$, and $D_1D_2^{-1}$. The images of these elements under the map~\eqref{eqn:relateambientringsS04} satisfy no relations, so we must have $P=0$ in~$(\mathcal{D}_{q,\mathbf{z}}^0)^{\mathbb{C}^*}$. Hence the map~\eqref{eqn:relateambientringsS04} is injective. It induces a $\mathbb{C}$-algebra homomorphism on the localization $\mathcal{D}_{q,\mathbf{z}}^{\mathbb{C}^*}/(z_{k,1}z_{k,2}-1,w_1w_2-1)$. Since a fraction maps to zero if and only if its numerator maps to zero, the map on the localization is also injective.
\end{proof}

Recall that in~Section~\ref{sec:GeneratingTheCoulombBranch} we defined $\mathcal{A}_{q,\mathbf{z}}$ to be the algebra generated by all minuscule monopole operators and all symmetric polynomials in the~$w_r$. Since our quiver has only one gauge node, we use only one index in the notation for the minuscule monopole operators. Thus, these operators will be written as $E_n[f]$, $F_n[f]$ for an integer $n\in\{1,2\}$ and a symmetric Laurent polynomial~$f$ in $n$ variables.

\begin{lemma}
\label{lem:explicitembeddingS04}
There exists an injective $\mathbb{C}$-algebra homomorphism 
\[
\Sk_{A,\boldsymbol{\lambda}}(S_{0,4})\hookrightarrow\mathcal{A}_{q,\mathbf{z}}^{\mathbb{C}^*}/(z_{k,1}z_{k,2}-1,w_1w_2-1)
\]
mapping $A\mapsto q^{-\frac{1}{2}}$, $\lambda_1\mapsto-z_{2,1}$, $\lambda_2\mapsto z_b$, $\lambda_3\mapsto z_a$, $\lambda_4\mapsto-z_{1,1}$, and 
\begin{equation}
\label{eqn:assignmentsS04}
\begin{split}
\alpha &\mapsto w_1+w_2, \\
\beta &\mapsto q^4z_a^{-1}z_b^{-1}\cdot F_1[x^2]E_1[1]+C_\beta, \\
\gamma &\mapsto q^3z_a^{-1}z_b^{-1}\cdot F_1[x]E_1[1]+C_\gamma
\end{split}
\end{equation}
for some $C_\beta$,~$C_\gamma\in\mathbb{C}[q^{\pm\frac{1}{2}},z_{k,l}^{\pm1},z_a^{\pm1},z_b^{\pm1},w_r]$. If we write $\gamma_n$ for the element defined in Figure~\ref{fig:S04curves}, then $\gamma_n\mapsto q^{4-n}z_a^{-1}z_b^{-1}\cdot F_1[x^{2-n}]E_1[1]+C_n$ for some $C_n\in\mathbb{C}[q^{\pm\frac{1}{2}},z_{k,l}^{\pm1},z_a^{\pm1},z_b^{\pm1},w_r]$.
\end{lemma}

\begin{proof}
Let $C_\beta$ and $C_\gamma$ be the unique elements of $\mathbb{C}[q^{\pm\frac{1}{2}},z_{k,l}^{\pm1},z_a^{\pm1},z_b^{\pm1},w_r]$ whose images under the embedding of Lemma~\ref{lem:relateambientringsS04} are the expressions $f_y(X)$ and $f_z(X)$, respectively, from Proposition~\ref{prop:polyrepS04}. We write $\mathcal{O}_1$,~$\mathcal{O}_2$,~$\mathcal{O}_3\in\mathcal{D}_{q,\mathbf{z}}^{\mathbb{C}^*}/(z_{k,1}z_{k,2}-1,w_1w_2-1)$ for the operators defined by the formulas on the right hand side of~\eqref{eqn:assignmentsS04} in order, and write $\mathcal{L}\subset\mathcal{D}_{q,\mathbf{z}}^{\mathbb{C}^*}/(z_{k,1}z_{k,2}-1,w_1w_2-1)$ for the $\mathbb{C}[q^{\pm\frac{1}{2}},z_{k,l}^{\pm1},z_a^{\pm1},z_b^{\pm1}]$-subalgebra generated by these operators. Similarly, we write $\mathcal{O}_\alpha$,~$\mathcal{O}_\beta$,~$\mathcal{O}_\gamma\in\End\mathbb{C}_{q,\mathbf{t}}(X)$ for the operators defined by the formulas on the right hand side of~\eqref{eqn:polyrepS04} in order, and write $\mathcal{M}\subset\End\mathbb{C}_{q,\mathbf{t}}(X)$ for the $\mathbb{C}_{q,\mathbf{t}}$-subalgebra generated by these operators. Now, if 
\[
\iota:\mathcal{D}_{q,\mathbf{z}}^{\mathbb{C}^*}/(z_{k,1}z_{k,2}-1,w_1w_2-1)\rightarrow\End\mathbb{C}_{q,\mathbf{t}}(X)
\]
is the embedding of Lemma~\ref{lem:relateambientringsS04}, then a straightforward calculation using Lemma~\ref{lem:relateambientringsS04} and the definition of the minuscule monopole operators shows that 
\[
\iota(\mathcal{O}_1)=\mathcal{O}_\alpha, \quad \iota(\mathcal{O}_2)=\mathcal{O}_\beta, \quad \iota(\mathcal{O}_3)=\mathcal{O}_\gamma,
\]
so we can view $\iota$ as an embedding $\iota:\mathcal{L}\hookrightarrow\mathcal{M}$. Since $\mathcal{M}$ preserves $\mathbb{C}_{q,\mathbf{t}}[X^{\pm1}]^{\mathbb{Z}_2}\subset\mathbb{C}_{q,\mathbf{t}}(X)$, we can view $\mathcal{M}$ as a subset of $\End\mathbb{C}_{q,\mathbf{t}}[X^{\pm1}]^{\mathbb{Z}_2}$ and thus view $\mathcal{L}$ as a subset of $\End\mathbb{C}_{q,\mathbf{t}}[X^{\pm1}]^{\mathbb{Z}_2}$. On~the other hand, by Proposition~\ref{prop:polyrepS04}, there is an embedding $\Sk_{A,\boldsymbol{\lambda}}(S_{0,4})\rightarrow\End\mathbb{C}_{q,\mathbf{t}}[X^{\pm1}]^{\mathbb{Z}_2}$ that maps a generator $x\in\{\alpha,\beta,\gamma\}$ to $\mathcal{O}_x$. Since the elements $\alpha$, $\beta$, and $\gamma$ generate $\Sk_{A,\boldsymbol{\lambda}}(S_{0,4})$ as a $\mathbb{C}[A^{\pm1},\lambda_1^{\pm1},\dots,\lambda_4^{\pm1}]$-algebra, we get an embedding 
\[
\Sk_{A,\boldsymbol{\lambda}}(S_{0,4})\hookrightarrow\mathcal{L}\hookrightarrow\mathcal{A}_{q,\mathbf{z}}^{\mathbb{C}^*}/(z_{k,1}z_{k,2}-1,w_1w_2-1)
\]
with the required values on generators. Let $C_n$ be the unique element of $\mathbb{C}[q^{\pm\frac{1}{2}},z_{k,l}^{\pm1},z_a^{\pm1},z_b^{\pm1},w_r]$ whose image under $\iota$ is the expression $f_n(X)$ from Lemma~\ref{lem:imageS04curves}. Then, in the same way as before, one checks that $\iota\left(q^{4-n}z_a^{-1}z_b^{-1}\cdot F_1[x^{2-n}]E_1[1]+C_n\right)$ is the operator appearing in Lemma~\ref{lem:imageS04curves}. This proves the second statement.
\end{proof}

\begin{lemma}
\label{lem:S04surjective}
The embedding from Lemma~\ref{lem:explicitembeddingS04} is a $\mathbb{C}$-algebra isomorphism 
\[
\Sk_{A,\boldsymbol{\lambda}}(S_{0,4})\cong\mathcal{A}_{q,\mathbf{z}}^{\mathbb{C}^*}/(z_{k,1}z_{k,2}-1,w_1w_2-1).
\]
\end{lemma}

\begin{proof}
Let us abbreviate $\mathcal{A}_{q,\mathbf{z}}'\coloneqq\mathcal{A}_{q,\mathbf{z}}^{\mathbb{C}^*}/(z_{k,1}z_{k,2}-1,w_1w_2-1)$. This algebra is generated by dressed minuscule monopole operators together with all symmetric polynomials in the~$w_i$. Using the definition of the dressed minuscule monopole operators, one can check that the products $E_1[x^m]F_2[x^n]$, $F_2[x^n]E_1[x^m]$,$E_2[x^m]F_1[x^n]$, and $F_1[x^n]E_2[x^m]$ can be written as linear combinations of dressed minuscule monopole operators of the form $E_1[x^k]$ and $F_1[x^k]$. The $\mathbb{C}^*$-invariant products of operators of the form $E_2[f]$ and $F_2[f]$ are simply symmetric polynomials in the $w_i$. These facts imply that the algebra $\mathcal{A}_{q,\mathbf{z}}'$ is in fact generated by dressed minuscule monopole operators of the form $E_1[x^k]$ and~$F_1[x^k]$, together with all symmetric polynomials in the~$w_i$.

In the algebra $\mathcal{A}_{q,\mathbf{z}}$, it is easy to show that we have the relations 
\begin{align}
\label{eqn:commutepolynomial}
\begin{split}
\left[E_1[x^m],w_1^{\pm1}+w_2^{\pm1}\right] &= (q^{\pm2}-1)E_1[x^{m\pm1}], \\
\left[F_1[x^m],w_1^{\pm1}+w_2^{\pm1}\right] &= (1-q^{\pm2})F_1[x^{m\pm1}].
\end{split}
\end{align}
By the results of~\cite{FT19b}, in particular their equations~(3.34) and~(3.39), we have a relation 
\begin{equation}
\label{eqn:FT}
\left[E_1[x^m],F_1[x^n]\right]=(q-q^{-1})\cdot h
\end{equation}
for some symmetric polynomial $h$ in the~$w_i$. (Our operators $E_1[x^m]$ and $F_1[x^n]$ correspond to the operators $F_{2,1}^{(m)}$ and $E_{1,2}^{(n-2)}$, respectively, in~\cite{FT19b}). Finally, in $\mathcal{A}_{q,\mathbf{z}}'$ we have the relations 
\begin{align}
\label{eqn:changedressing}
\begin{split}
E_1[x^m]F_1[x^n] &= (w_1^m+w_2^m)E_1[1]F_1[x^n]-E_1[x^{-m}]F_1[x^n], \\
F_1[x^n]E_1[x^m] &= F_1[x^{n+1}]E_1[x^{m-1}]+q^{-2}F_1[x^{n-1}]E_1[x^{m-1}]-q^{-2}F_1[x^n]E_1[x^{m-2}],
\end{split}
\end{align}
which can be checked directly. 

Let us consider an element $p\in\mathcal{A}_{q,\mathbf{z}}'$ which is a product of dressed minuscule monopole operators of the form $E_1[f]$ and~$F_1[f]$ and symmetric polynomials in the~$w_i$. Note that any symmetric polynomial in the~$w_i$ is a linear combination of polynomials of the form $w_1^n+w_2^n$, and the latter can be written as $P_n(w_1+w_2)$ where $P_n$ is a polynomial of degree~$n$. Using~\eqref{eqn:commutepolynomial} and~\eqref{eqn:FT}, we can write~$p$ as a linear combination of products of expressions of the form~$h_1\cdot E_1[x^m]F_1[x^n]\cdot h_2$ where~$h_1$ and~$h_2$ are symmetric polynomials in the~$w_i$. We can then use the first relation of~\eqref{eqn:changedressing} to write~$p$ as a linear combination of products of expressions of the form~$h_1\cdot E_1[x^m]F_1[x^n]\cdot h_2$ where~$m\geq0$. Finally, we can use~\eqref{eqn:FT} and the second relation of~\eqref{eqn:changedressing} to write~$p$ as a linear combination of products of expressions of the form $h_1\cdot F_1[x^n]E_1[1]\cdot h_2$ and $h_1\cdot F_1[x^n]E_1[x^{-1}]\cdot h_2$.

We claim that $F_1[x^n]E_1[x^{-1}]$ is in the image of the embedding $\Sk_{A,\boldsymbol{\lambda}}(S_{0,4})\hookrightarrow\mathcal{A}_{q,\mathbf{z}}'$ of Lemma~\ref{lem:explicitembeddingS04}. Indeed, by applying the two relations~\eqref{eqn:commutepolynomial} in order one obtains 
\begin{align*}
F_1[x^n]E_1[x^{-1}]=\frac{1}{q^{-2}-1}\bigl(F_1[x^n]E_1[1](w_1^{-1}+w_2^{-1}) &- (w_1^{-1}+w_2^{-1})F_1[x^n]E_1[1] \\
&- (1-q^{-2})F_1[x^{n-1}]E_1[1]\bigr)
\end{align*}
in the localized algebra $\mathcal{A}_{q,\mathbf{z}}'\otimes\mathbb{C}(q)$. The embedding of Lemma~\ref{lem:explicitembeddingS04} induces an embedding of localized algebras $\Sk_{A,\boldsymbol{\lambda}}(S_{0,4})\otimes\mathbb{C}(A)\hookrightarrow\mathcal{A}_{q,\mathbf{z}}'\otimes\mathbb{C}(q)$. It maps $\gamma_n\mapsto B_n\cdot F_1[x^{2-n}]E_1[1]+C_n$ where $B_n=q^{4-n}z_a^{-1}z_b^{-1}$ and $C_n\in\mathbb{C}[q^{\pm\frac{1}{2}},z_{k,l}^{\pm1},z_a^{\pm1},z_b^{\pm1},w_r]$. Let us write $Q_n$ and~$R_n$ for the unique elements of $\Sk_{A,\boldsymbol{\lambda}}(S_{0,4})$ whose images under this embedding are $B_n$ and $C_n$, respectively. Then this embedding maps the element 
\[
\frac{Q_{-n+2}^{-1}}{A^4-1}\bigl(\gamma_{-n+2}\alpha-\alpha\gamma_{-n+2} - (1-A^4)A^{-2}\gamma_{-n+3}\bigr) - A^{-2}R_{-n+3}Q_{-n+2}^{-1}
\]
to $F_1[x^n]E_1[x^{-1}]$. For any index~$m$, the skein relation implies $\alpha\gamma_m=A^2\gamma_{m+1}+A^{-2}\gamma_{m-1}+T_m$ for some $T_m\in\mathbb{C}[A^{\pm1},\lambda_1^{\pm1},\dots,\lambda_4^{\pm1}]$. Similarly, we have $\gamma_m\alpha=A^2\gamma_{m-1}+A^{-2}\gamma_{m+1}+T_m$. Thus 
\[
\gamma_m\alpha-\alpha\gamma_m=(A^2-A^{-2})\gamma_{m-1}+(A^{-2}-A^2)\gamma_{m+1}
\]
is divisible by $A^4-1$. This implies the claim.

We thus see that $F_1[x^n]E_1[x^{-1}]$ is in the image of the embedding $\Sk_{A,\boldsymbol{\lambda}}(S_{0,4})\hookrightarrow\mathcal{A}_{q,\mathbf{z}}'$. Lemma~\ref{lem:explicitembeddingS04} says that $F_1[x^n]E_1[1]$ is also in the image of this embedding. Hence the element~$p$ considered above is in the image of this embedding. This proves that the embedding is surjective and hence an isomorphism.
\end{proof}

\begin{theorem}
\label{thm:S04main}
Let $(\widetilde{G},N)$ be the group and representation associated to a pants decomposition of the four-holed sphere~$S_{0,4}$. Then there is a $\mathbb{C}$-algebra isomorphism 
\[
\Sk_{A,\boldsymbol{\lambda}}(S_{0,4})\cong K^{\widetilde{G}_{\mathcal{O}}\rtimes\mathbb{C}^*}(\mathcal{R}_{G,N}).
\]
\end{theorem}

\begin{proof}
Any pants decomposition of~$S_{0,4}$ is obtained from a disjoint union $P_1\sqcup P_2$ of pairs of pants $P_1$ and~$P_2$ by identifying a boundary component of~$P_1$ with a boundary component of~$P_2$. Thus, from the construction of Section~\ref{sec:RepresentationsFromSurfaces}, we see that the corresponding gauge group is $G=\mathrm{SL}_2(\mathbb{C})$. The representation is 
\[
N=(\mathbb{C}^2\otimes\mathbb{C}^2)\oplus(\mathbb{C}^2\otimes\mathbb{C}^2)\cong\Hom(\mathbb{C}^2,\mathbb{C}^2)\oplus\Hom(\mathbb{C}^2,\mathbb{C}^2)
\]
with the $G$-action defined by $g\cdot(f_1,f_2)=(g\circ f_1,g\circ f_2)$ for any element $g\in G$ and linear maps $f_1$,~$f_2:\mathbb{C}^2\rightarrow\mathbb{C}^2$. Let $F_1$ and $F_2$ be two copies of the diagonal subgroup of $\mathrm{SL}_2(\mathbb{C})$. Then there is an action of the product $F_1\times F_2$ on~$N$ defined by $(h_1,h_2)\cdot(f_1,f_2)=(f_1\circ h_1^{-1},f_2\circ h_2^{-1})$ for $h_i\in F_i$. If we let $F_3$ and $F_4$ be copies of $\mathbb{C}^*$, then there is a further action of $F_3\times F_4$ on~$N$ by coordinatewise rescaling. The flavor symmetry group is the product $F=F_1\times F_2\times F_3\times F_4\cong(\mathbb{C}^*)^4$.

We also have an action of $G'=\mathrm{GL}_2(\mathbb{C})$ on~$N$ given by $g\cdot(f_1,f_2)=(g\circ f_1,g\circ f_2)$ for~$g\in G'$. If~$F_1'$ and $F_2'$ are two copies of the diagonal subgroup of~$\mathrm{GL}_2(\mathbb{C})$, then there is an action of $F_1'\times F_2'$ on~$N$ by $(h_1,h_2)\cdot(f_1,f_2)=(f_1\circ h_1^{-1},f_2\circ h_2^{-1})$ for $h_i\in F_i'$. If we set $\widetilde{G}'=G'\times F'$ where $F'=F_1'\times F_2'\times F_3\times F_4$, then $(\widetilde{G}',N)$ is precisely the pair arising from the quiver~\eqref{eqn:S04quiver}. According to Proposition~\ref{prop:generate}, the embedding constructed in Section~\ref{sec:TheAbelianCase} identifies the Coulomb branch of $(\widetilde{G}',N)$ with the algebra~$\mathcal{A}_{q,\mathbf{z}}$. It therefore follows from Lemma~\ref{lem:KSL2GL2} that 
\[
\mathcal{A}_{q,\mathbf{z}}^{\mathbb{C}^*}/(z_{k,1}z_{k,2}-1,w_1w_2-1)\cong K^{\widetilde{G}_{\mathcal{O}}\rtimes\mathbb{C}^*}(\mathcal{R}_{G,N}).
\]
Hence the desired statement follows from Lemma~\ref{lem:S04surjective}.
\end{proof}

\begin{remark}
After the first version of this paper appeared online, Hiraku~Nakajima informed us of related work in progress with Michael~Finkelberg, which uses monopole operators to establish a relationship between the $K$-theoretic Coulomb branches of certain quiver gauge theories and DAHA of type $(C_n^\vee,C_n)$. Their work was inspired by an unpublished calculation of Yutaka~Yoshida in the $n=1$ case.
\end{remark}

\subsection{The one-holed torus}

To understand the Coulomb branch of the representation associated to a one-holed torus, we consider the Jordan quiver: 
\begin{equation}
\label{eqn:Jordanquiver}
\xymatrix{
*+[o][F-]{2} \ar@(ur,ul)
}
\end{equation}
By applying the construction of Section~\ref{sec:MonopoleOperators} to this quiver, we get $\mathbb{C}[q^{\pm\frac{1}{2}},z_a^{\pm1}]$-algebras~$\mathcal{D}_{q,\mathbf{z}}^0$ and~$\mathcal{D}_{q,\mathbf{z}}$ where $a$ is the arrow of the quiver. Since there is only one arrow, we will abbreviate $z=z_a$ and denote these algebras as $\mathcal{D}_{q,z}^0$ and~$\mathcal{D}_{q,z}$, respectively. These algebras are generated by variables~$w_r$, $D_r$, and their inverses where $r\in\{1,2\}$. Here again we write only the $r$ index since the quiver has only one gauge node. There are $\mathbb{C}^*$-actions on~$\mathcal{D}_{q,z}^0$ and~$\mathcal{D}_{q,z}$ that simultaneously rescale the~$D_r$, and we will write $(\mathcal{D}_{q,z}^0)^{\mathbb{C}^*}$ and~$\mathcal{D}_{q,z}^{\mathbb{C}^*}$ for the corresponding invariant subalgebras.

\begin{lemma}
\label{lem:relateambientringsS11}
Define the operator $\varpi\in\End\mathbb{C}_{q,t}(X)$ by $(\varpi f)(X)=f(qX)$ for $f\in\mathbb{C}_{q,t}(X)$. Then the assignments 
\begin{gather*}
q^{\frac{1}{2}}\mapsto q^{\frac{1}{2}}, \quad z\mapsto q^{-1}t, \quad w_1\mapsto X, \quad w_2\mapsto X^{-1}, \quad D_1D_2^{-1}\mapsto q^{-4}X^{-4}\varpi^2
\end{gather*}
determine a well defined injective $\mathbb{C}$-algebra homomorphism 
\[
\mathcal{D}_{q,z}^{\mathbb{C}^*}/(w_1w_2-1)\hookrightarrow\End\mathbb{C}_{q,t}(X).
\]
\end{lemma}

\begin{proof}
One can check that these assignments map the commutation relations in~$(\mathcal{D}_{q,z}^0)^{\mathbb{C}^*}$ to relations in $\End\mathbb{C}_{q,t}(X)$ and therefore provide a well defined $\mathbb{C}$-algebra homomorphism $(\mathcal{D}_{q,z}^0)^{\mathbb{C}^*}\rightarrow\End\mathbb{C}_{q,t}(X)$. The element $w_1w_2-1$ maps to zero, so this induces a well defined $\mathbb{C}$-algebra homomorphism 
\begin{equation}
\label{eqn:relateambientringsS11}
(\mathcal{D}_{q,z}^0)^{\mathbb{C}^*}/(w_1w_2-1)\rightarrow\End\mathbb{C}_{q,t}(X).
\end{equation}
Suppose $\bar{P}$ lies in the kernel of this map. This $\bar{P}$ is the class of some $P\in(\mathcal{D}_{q,z}^0)^{\mathbb{C}^*}$, which we can take to be a Laurent polynomial in $q^{\frac{1}{2}}$, $z$, $w_1$, and $D_1D_2^{-1}$. The images of these elements under the map~\eqref{eqn:relateambientringsS11} satisfy no relations, so we must have $P=0$ in~$(\mathcal{D}_{q,z}^0)^{\mathbb{C}^*}$. Hence the map~\eqref{eqn:relateambientringsS11} is injective. It induces a $\mathbb{C}$-algebra homomorphism on the localization $\mathcal{D}_{q,z}^{\mathbb{C}^*}/(w_1w_2-1)$. Since a fraction maps to zero if and only if its numerator maps to zero, the map on the localization is also injective.
\end{proof}

Let us now write $\mathcal{A}_{q,z}=\mathcal{A}_{q,\mathbf{z}}$ as in~Section~\ref{sec:GeneratingTheCoulombBranch} for the algebra generated by all minuscule monopole operators and all symmetric polynomials in the~$w_r$. Since our quiver has only one vertex, we again use only one index in the notation for the minuscule monopole operators. We also consider the relative skein algebra $\Sk_{A,\lambda}(S_{1,1})$ of the one-holed torus. It has a $\mathbb{Z}_2$-action where the nontrivial element acts on the generators of Proposition~\ref{prop:skeinS11} by $\alpha\mapsto\alpha$, $\beta\mapsto-\beta$, $\gamma\mapsto-\gamma$.

\begin{lemma}
\label{lem:explicitembeddingS11}
There exists an injective $\mathbb{C}$-algebra homomorphism $\Sk_{A,\lambda}(S_{1,1})^{\mathbb{Z}_2}\hookrightarrow\mathcal{A}_{q,z}^{\mathbb{C}^*}/(w_1w_2-1)$ mapping $A\mapsto q^{-\frac{1}{2}}$, $\lambda\mapsto z$, and
\begin{equation}
\label{eqn:assignmentsS11}
\begin{split}
\alpha &\mapsto w_1+w_2, \\
\beta^2 &\mapsto qz^{-1}\cdot F_1[1]E_1[1], \\
\gamma\beta &\mapsto q^{\frac{5}{2}}z^{-1}\cdot F_1[x]E_1[1], \\
\gamma^2 &\mapsto qz^{-1}\cdot F_1[x]E_1[x^{-1}].
\end{split}
\end{equation}
If we write $\gamma_n$ for the element in Figure~\ref{fig:S11curves}, then $\gamma_n\beta\mapsto q^{(3n+2)/2}z^{-1}\cdot F_1[x^n]E_1[1]$ and $\gamma_n\gamma\mapsto q^{(3n-1)/2}z^{-1}\cdot F_1[x^n]E_1[x^{-1}]$.
\end{lemma}

\begin{proof}
Let $\mathcal{O}_1,\dots,\mathcal{O}_4\in\mathcal{D}_{q,z}^{\mathbb{C}^*}/(w_1w_2-1)$ be the operators defined by the formulas on the right hand side of~\eqref{eqn:assignmentsS11} in order, and let $\mathcal{L}\subset\mathcal{D}_{q,z}^{\mathbb{C}^*}/(w_1w_2-1)$ be the $\mathbb{C}[q^{\pm\frac{1}{2}},z^{\pm1}]$-subalgebra generated by these operators. Similarly, let $\mathcal{O}_\alpha$,~$\mathcal{O}_\beta$,~$\mathcal{O}_\gamma\in\End\mathbb{C}_{q,t}(X)$ be the operators defined by the formulas on the right hand side of~\eqref{eqn:polyrepS11} in order, and let $\mathcal{M}\subset\End\mathbb{C}_{q,t}(X)$ be the $\mathbb{C}_{q,t}$-subalgebra generated by these operators. Let us write $\iota:\mathcal{D}_{q,z}^{\mathbb{C}^*}/(w_1w_2-1)\rightarrow\End\mathbb{C}_{q,t}(X)$ for the embedding of Lemma~\ref{lem:relateambientringsS11}. Then a straightforward calculation using Lemma~\ref{lem:relateambientringsS11} and the definition of the minuscule monopole operators shows that 
\[
\iota(\mathcal{O}_1)=\mathcal{O}_\alpha, \quad \iota(\mathcal{O}_2)=\mathcal{O}_\beta^2, \quad \iota(\mathcal{O}_3)=\mathcal{O}_\gamma\mathcal{O}_\beta, \quad \iota(\mathcal{O}_4)=\mathcal{O}_\gamma^2,
\]
so we can view $\iota$ as an embedding $\iota:\mathcal{L}\hookrightarrow\mathcal{M}$. Since $\mathcal{M}$ preserves $\mathbb{C}_{q,t}[X^{\pm1}]^{\mathbb{Z}_2}\subset\mathbb{C}_{q,t}(X)$, we can view $\mathcal{M}$ as a subset of $\End\mathbb{C}_{q,t}[X^{\pm1}]^{\mathbb{Z}_2}$ and thus view $\mathcal{L}$ as a subset of $\End\mathbb{C}_{q,t}[X^{\pm1}]^{\mathbb{Z}_2}$. On~the other hand, by Proposition~\ref{prop:polyrepS11}, there is an embedding $\Sk_{A,\lambda}(S_{1,1})\rightarrow\End\mathbb{C}_{q,t}[X^{\pm1}]^{\mathbb{Z}_2}$ that maps a generator $x\in\{\alpha,\beta,\gamma\}$ to $\mathcal{O}_x$. Since the elements $\alpha$, $\beta^2$, $\gamma\beta$, and $\gamma^2$ generate $\Sk_{A,\lambda}(S_{1,1})^{\mathbb{Z}_2}$ as a $\mathbb{C}[A^{\pm1},\lambda^{\pm1}]$-algebra, we get an embedding 
\[
\Sk_{A,\lambda}(S_{1,1})^{\mathbb{Z}_2}\hookrightarrow\mathcal{L}\hookrightarrow\mathcal{A}_{q,z}^{\mathbb{C}^*}/(w_1w_2-1)
\]
with the required values on generators. In the same way, one can check that $\iota(q^{(3n+2)/2}z^{-1}\cdot F_1[x^n]E_1[1])$ equals the product of the operator in Lemma~\ref{lem:imageS11curves} and the second operator in~\eqref{eqn:polyrepS11} and that $\iota(q^{(3n-3)/2}z^{-1}\cdot F_1[x^n]E_1[x^{-1}])$ equals the product of the operator in Lemma~\ref{lem:imageS11curves} and the third operator in~\eqref{eqn:polyrepS11}. This proves the second statement.
\end{proof}

\begin{lemma}
\label{lem:S11surjective}
The embedding from Lemma~\ref{lem:explicitembeddingS11} is a $\mathbb{C}$-algebra isomorphism 
\[
\Sk_{A,\lambda}(S_{1,1})^{\mathbb{Z}_2}\cong\mathcal{A}_{q,z}^{\mathbb{C}^*}/(w_1w_2-1).
\]
\end{lemma}

\begin{proof}
Let us abbreviate $\mathcal{A}_{q,z}'\coloneqq\mathcal{A}_{q,z}^{\mathbb{C}^*}/(w_1w_2-1)$. Then any element of~$\mathcal{A}_{q,z}'$ can be written as a linear combination of products of expressions of the form $h_1\cdot F_1[x^n]E_1[1]\cdot h_2$ and $h_1\cdot F_1[x^n]E_1[x^{-1}]\cdot h_2$ where~$h_1$ and~$h_2$ are symmetric polynomials in the variables~$w_i$. This follows exactly as in the proof of Lemma~\ref{lem:S04surjective}, except that instead of equation~\eqref{eqn:FT}, one uses the more explicit commutation relation $E_1[x^m]F_1[x^n]=q^{-2(m+n)}F_1[x^{-m}]E_1[x^{-n}]$, which can be checked directly using the formulas for the dressed minuscule monopole operators. It follows from Lemma~\ref{lem:explicitembeddingS11} that every expression of the form $h_1\cdot F_1[x^n]E_1[1]\cdot h_2$ or $h_1\cdot F_1[x^n]E_1[x^{-1}]\cdot h_2$ is in the image of the embedding $\Sk_{A,\lambda}(S_{1,1})^{\mathbb{Z}_2}\hookrightarrow\mathcal{A}_{q,z}'$. Hence this embedding is an isomorphism as claimed.
\end{proof}

\begin{theorem}
\label{thm:S11main}
Let $(\widetilde{G},N)$ be the group and representation associated to a pants decomposition of the one-holed torus~$S_{1,1}$. Then there is a $\mathbb{Z}_2$-action on $\Sk_{A,\lambda}(S_{1,1})$ and a $\mathbb{C}$-algebra isomorphism 
\[
\Sk_{A,\lambda}(S_{1,1})^{\mathbb{Z}_2}\cong K^{\widetilde{G}_{\mathcal{O}}\rtimes\mathbb{C}^*}(\mathcal{R}_{G,N}).
\]
\end{theorem}

\begin{proof}
Any pants decomposition of~$S_{1,1}$ is obtained by gluing together two boundary components of a single pair of pants. From the construction of Section~\ref{sec:RepresentationsFromSurfaces}, we see that the corresponding gauge group is $G=\mathrm{SL}_2(\mathbb{C})$ and that $N=\mathbb{C}^2\otimes\mathbb{C}^2$ is the bifundamental representation of~$G$. The flavor symmetry group $F=\mathbb{C}^*$ acts by rescaling vectors in~$N$. We have $N\cong\Hom(\mathbb{C}^2,\mathbb{C}^2)$ as representations of $\widetilde{G}=G\times F$ where the factor $G$ acts on~$\Hom(\mathbb{C}^2,\mathbb{C}^2)$ by conjugation and $F$ acts by rescaling. We also have an action of $G'=\mathrm{GL}_2(\mathbb{C})$ on~$N$ by conjugation, which extends to an action of $\widetilde{G}'=G'\times F$. Then $(\widetilde{G}',N)$ is precisely the pair associated to the quiver~\eqref{eqn:Jordanquiver}. According to Proposition~\ref{prop:generate}, the embedding constructed in Section~\ref{sec:TheAbelianCase} identifies the Coulomb branch of~$(\widetilde{G}',N)$ with~$\mathcal{A}_{q,\mathbf{z}}$. It follows from Lemma~\ref{lem:KSL2GL2} that 
\[
\mathcal{A}_{q,z}^{\mathbb{C}^*}/(w_1w_2-1)\cong K^{\widetilde{G}_{\mathcal{O}}\rtimes\mathbb{C}^*}(\mathcal{R}_{G,N}),
\]
and hence the desired statement follows from Lemma~\ref{lem:S11surjective}.
\end{proof}

\subsection{Evidence for the main conjecture}
\label{sec:EvidenceFortheMainConjecture}

Based on our observations in Theorems~\ref{thm:S03main}, \ref{thm:S04main}, and~\ref{thm:S11main}, we propose that skein algebras and quantized Coulomb branches are related as in Conjecture~\ref{conj}. As further evidence for this conjecture, we note that skein algebras and quantized $K$-theoretic Coulomb~branches admit similar-looking canonical bases. For the skein algebra of a surface~$S$, one has three closely related canonical bases, each of which is parametrized by multicurves on~$S$; see~\cite{T14} and also~\cite{B23} for some examples considered in the present paper. Given a fixed pants~decomposition of~$S$, one can describe multicurves on~$S$ up to isotopy by integer coordinates known as Dehn--Thurston coordinates~\cite{PH92}. The Dehn--Thurston coordinates of a multicurve consist of a twist coordinate and a nonnegative intersection number for each curve in the pants decomposition.

On the other hand, Cautis and Williams~\cite{CW23} have introduced canonical bases for quantized $K$-theoretic Coulomb~branches. For a theory with gauge group~$G$, the canonical basis for its quantized Coulomb~branch is parametrized by the set $(\Lambda\times\Lambda^\vee)/W$ where $\Lambda$ and $\Lambda^\vee$ are the weight and coweight lattices of~$G$ and $W$ is the Weyl~group acting diagonally. An element of this set determines a vector bundle over a $G_{\mathcal{O}}$-orbit $\Gr_G^\lambda$ in the affine Grassmannian $\Gr_G$, which in turn gives rise to an element of the quantized Coulomb branch.

When $G$ is the gauge group associated to a pants decomposition of the surface $S_{g,n}$, we have $(\Lambda\times\Lambda^\vee)/W=\left(\mathbb{Z}^m\times(2\mathbb{Z})^m\right)/\{\pm1\}$ where $\{\pm1\}$ acts diagonally and $m=3g-3+n$ is the number of curves in the pants decomposition. A beautiful fact, first observed in~\cite{DMO09}, is that an element of the latter set naturally determines a collection of Dehn--Thurston coordinates for a multicurve. Thus it should be possible to explicitly match the canonical bases under the isomorphisms appearing in Conjecture~\ref{conj}.

We note that the $\mathbb{Z}_2$-action in Theorem~\ref{thm:S11main} is the one where the nontrivial element acts on generators of the skein algebra by $\alpha\mapsto\alpha$, $\beta\mapsto-\beta$, $\gamma\mapsto-\gamma$. More generally, in part~\eqref{conj:genus1} of Conjecture~\ref{conj}, we expect that the nontrivial element of~$\mathbb{Z}_2$ maps $\beta\mapsto-\beta$ where $\beta$ is the element of~$\Sk_{A,\boldsymbol{\lambda}}(S)$ corresponding to a simple closed curve that intersects each of the curves in Figure~\ref{subfig:genus1} exactly once. This ensures that the Dehn--Thurston coordinates of the projection of any framed link in $\Sk_{A,\boldsymbol{\lambda}}(S)^{\mathbb{Z}_2}$ lie in the set $(\mathbb{Z}^m\times(2\mathbb{Z}))/\{\pm1\}$.

\bibliographystyle{amsplain}

\end{document}